\documentclass{amsart}

\usepackage{delimset, amssymb, enumitem, caption, mathtools}
\setlist{itemsep=4pt, topsep=0pt, leftmargin=17pt, listparindent=11pt}

\usepackage{xcolor}
\definecolor{PKU}{cmyk}{0, 1, 1, .45}
\definecolor{BIT}{cmyk}{1, 0, 1, 0}
\usepackage[colorlinks, citecolor=BIT, linkcolor=PKU, pagebackref, ocgcolorlinks]{hyperref}

\usepackage[capitalize]{cleveref}
\crefname{def}{Def.}{Defs.}
\Crefname{def}{Definition}{Definitions}
\creflabelformat{def}{~\upshape(#2#1#3)}
\crefname{ineq}{Ineq.}{Ineqs.}
\Crefname{ineq}{Inequality}{Inequalities}
\creflabelformat{ineq}{~\upshape(#2#1#3)}
\crefname{itm}{}{}
\creflabelformat{itm}{~\upshape(#2#1#3)}
\crefname{conjecture}{Conjecture}{Conjectures}
\crefname{sec}{\S}{\S}
\Crefname{sec}{Section}{Sections}
\creflabelformat{sec}{\upshape#2#1#3}

\newtheorem{theorem}{Theorem}[section]
\newtheorem{corollary}[theorem]{Corollary}
\newtheorem{conjecture}[theorem]{Conjecture}
\newtheorem{lemma}[theorem]{Lemma}
\newtheorem{proposition}[theorem]{Proposition}

\numberwithin{equation}{section}

\allowbreak
\allowdisplaybreaks

\usepackage[numbers, sort&compress, nonamebreak, merge, elide, longnamesfirst]{natbib}

\usepackage{url}

\usepackage{tikz}
\usetikzlibrary{calc, positioning}
\tikzset{
edge/.style={semithick},
ball/.style={shape=circle, minimum size=1mm, ball color=black, inner sep=0.5},
ellipsis/.style={shape=circle, fill, inner sep=.5}}
\def\r{1}
\def\eps{.2}

\title[Neat formulas for chromatic symmetric functions]
{A composition method for neat formulas of chromatic symmetric functions}

\author[D.G.L.~Wang]{David G.L. Wang$^*$}
\address[David G.L. Wang]{School of Mathematics and Statistics \& MIIT Key Laboratory of Mathematical Theory and Computation in Information Security, Beijing Institute of Technology, Beijing 102400, P. R. China.}
\email{glw@bit.edu.cn}
\thanks{$^*$Wang is the corresponding author, and is supported by the NSFC (Grant No.\ 12171034).}

\author[J.Z.F. Zhou]{James Z.F. Zhou}
\address[James Z.F. Zhou]{School of Mathematics and Statistics, Beijing Institute of Technology, Beijing 102400, P. R. China.}
\email{james@bit.edu.cn}


\begin{document}

\bibliographystyle{abbrvnat}

\begin{abstract}
We develop a composition method to unearth positive $e_I$-expansions of chromatic symmetric functions $X_G$, where the subscript $I$ stands for compositions rather than integer partitions. Using this method, we derive positive and neat $e_I$-expansions for the chromatic symmetric functions of tadpoles, barbells and generalized bulls, and establish the $e$-positivity of hats. We also obtain a compact ribbon Schur analog for the chromatic symmetric function of cycles. 
\end{abstract}
\maketitle
\tableofcontents

\section{Introduction}
 
In his seminal paper \cite{Sta95},
\citeauthor{Sta95} introduced
the concept of the chromatic symmetric function~$X_G$ for any graph $G$,
which tracks proper colorings of $G$.
It is a generalization of Birkhoff's chromatic
symmetric polynomial $\chi_G$ in the study of the $4$-color problem.
Chromatic symmetric functions encode 
many graph parameters and combinatorial structures, 
such like the number of vertices, edges and triangles,
the girth, and 
the lattice of contractions, 
see \citet{MMW08} and
\cite[Page~167]{Sta95}.
For any basis $b$ of the algebra $\mathrm{Sym}$ of symmetric functions,
a graph $G$ is said to be \emph{$b$-positive}
if every $b$-coefficient of $X_G$
is nonnegative.
\citet[Section~5]{Sta95}
brought forward the question that
which graphs are $e$-positive, 
and asserted that a complete characterization of $e$-positive
graphs ``appears hopeless.''
He restated \citeauthor{SS93}'s $(3+1)$-free conjecture \cite{SS93},
which became a leading conjecture in the study of chromatic symmetric functions henceforth.
\begin{conjecture}[\citeauthor{SS93}]\label{conj:epos}
The chromatic symmetric function of the incomparability graph
of every $(3+1)$-free poset is 
$e$-positive.
\end{conjecture}

\citet{Gas96P} confirmed the Schur positivity of the graphs 
in \cref{conj:epos},
which are all claw-free.
\citet{Sta98} then proposed 
the following Schur positivity conjecture 
and attributed it to \citeauthor{Gas96P},
see also \citet{Gas99}.
\begin{conjecture}[\citeauthor{Sta98} and \citeauthor{Gas99}]
\label{conj:spos}
Every claw-free graph is Schur positive.
\end{conjecture}
\citet{SW12} introduced the notion of chromatic quasisymmetric functions,
refined \citeauthor{Gas96P}'s Schur positivity result,
and unveiled connections between \cref{conj:epos} 
and representation theory.
By \citeauthor{Gua13X}'s reduction \cite{Gua13X},
\cref{conj:epos} can be restated as that
every unit interval graph, or equivalently,
every claw-free interval graph, is $e$-positive.
These conjectures thereby charm graph theorists that are fascinated by 
claw-free graphs and interval graphs,
see \citet{FFR97} for an early survey on claw-free graphs,
and \citet{COS10} for 
wide applications of interval graphs.
The Schur positivity of interval graphs can be shown
by using a result of \citet{Hai93}. 
\citeauthor{Hai93}'s proof used Kazhdan and Lusztig's conjectures
that were confirmed later, see~\cite[Page 187]{Sta95}.

Technically speaking, 
to show that a graph is not $e$-positive or not Schur positive 
is comparably undemanding, 
in the sense that the demonstration 
of a negative $e_\lambda$- or $s_\lambda$-coefficient
for a particular partition $\lambda$ is sufficient,
which may call for a scrupulous selection of $\lambda$ though.
For instance,
\citet{WW23-DAM} 
proved the non-$e$-positivity 
and non-Schur positivity of some spiders and brooms.
Two common criteria for the non-positivity are
\citeauthor{Wol97D}'s connected partition criterion
and \citeauthor{Sta98}'s stable partition criterion,
see \cite{Wol97D} and \cite{Sta98} respectively.

In contrast, to confirm that a graph is $e$-positive is seldom easy.
%Very recently, 
%\citet{Hik24X} confirmed \cref{conj:epos}
%with the aid of \citeauthor{Kato24X}'s new geometry realization 
%of the chromatic symmetric functions, see \cite{Kato24X}.
\citet{Sta95} studied paths and cycles 
by displaying the generating functions 
of their chromatic symmetric functions,
whose Taylor expansions indicate the $e$-positivity as plain sailing.
\citet{GS01} lifted $X_G$
up to certain~$Y_G$ in the algebra $\mathrm{NCSym}$ 
of symmetric functions in noncommutative variables,
so that $X_G$ equals the commutative image of~$Y_G$.
They developed a theory for certain $(e)$-positivity of $Y_G$,
which leads to the $e$-positivity of $X_G$.
In particular,
$K$-chains are $e$-positive.
\citet{Tom24} 
obtained an $e$-expansion of the chromatic symmetric function
of a general unit interval graph in terms of ``forest triples,''
and used it to reconfirm the $e$-positivity of $K$-chains.
\citet{Dv20} classified when~$Y_G$ 
is a positive linear combination 
of the elementary symmetric functions in noncommuting variables.
Via this $Y_G$-approach,
\citet{WW23-JAC} uncovered the $e$-positivity of two 
classes of cycle-chords.
\citet{AWv24} reinterpreted the equivalence idea 
for the $(e)$-positivity 
in terms of the quotient algebra $\mathrm{UBCSym}$ of $\mathrm{NCSym}$
and obtained the $e$-positivity of kayak paddle graphs.
An example of using chromatic quasisymmetric functions
to show the $e$-positivity can be found from \citet{HNY20}
for melting lollipops.

We think the plainest way of confirming the $e$-positivity of 
a graph $G$ is to compute $X_G$
out and make certain that the $e_\lambda$-coefficient 
for each partition $\lambda$ is nonnegative.
A variant idea is to recast~$X_G$
as a linear combination of $e$-positive chromatic symmetric functions
with positive coefficients,
see \citet{Dv18} for a treatment of lollipops for example.
Up to late~2023,
to the best of our knowledge, 
only complete graphs, paths, cycles, melting lollipops, $K$-chains,
and slightly melting $K$-chains
own explicit formulas of chromatic symmetric functions,
see \cref{sec:csf} and \citet{Tom24}.
In this paper,
we conceive a new approach along this way, called 
the \emph{composition method}.

We were inspired from \citeauthor{SW16}'s discovery 
\begin{equation}\label{X.path}
X_{P_n}
=\sum_{I=i_1 i_2\dotsm\vDash n}
w_I 
e_I
\end{equation}
for paths $P_n$, 
where the sum runs over compositions $I$ of $n$, and
\begin{equation}\label[def]{def:w}
w_I
=
i_1
\prod_{j\ge 2}
(i_j-1).
\end{equation}
They \cite[Table~1]{SW16} obtained \cref{X.path} by
using Stanley's generating function for Smirnov words,
see also \citet[Theorem 7.2]{SW10}. 
An equally engaging formula for cycles was brought to light by
\citet{Ell17}, see \cref{prop:cycle}.

The composition method
is to expand a chromatic symmetric function $X_G$
in the elementary symmetric functions $e_I$
which are indexed by compositions $I$.
This idea can be best understood through \cref{X.path}.
The $e_I$-coefficients,
taking \cref{def:w} for example,
are functions defined for compositions.
See \cref{sec:CompFn} for more examples.
An ordinary $e_\lambda$-coefficient 
for any partition $\lambda$ is 
the sum of ``the $e_I$-coefficients'' 
over all compositions~$I$
that can be rearranged as $\lambda$; 
we write this property of $I$ as 
\begin{equation}\label{def:rho}
\rho(I)=\lambda.
\end{equation}
Here arises a potential ambiguity
about the wording ``the $e_I$-coefficient''.
Namely,  
when the parts of~$I$ decrease weakly and 
so $I$ 
coincides with $\lambda$,
it
may be understood as 
either the coefficient of $e_I$ in some $e_I$-expansion
or the coefficient of $e_\lambda$ 
in the unique $e$-expansion of $X_G$.
This ambiguity comes from the unspecification
of the background algebra, which leads us to 
the algebra $\mathrm{NSym}$ of noncommutative symmetric functions,
see \cref{sec:sf,sec:NSym} for details.

In order to give a step by step instruction for applying the composition method, we need some basic knowledge of the algebra $\mathrm{NSym}$.
First, the commutative images of the basis elements 
$\Lambda^I$ and~$\Psi^I$
of $\mathrm{NSym}$
are the elementary and
power sum symmetric functions $e_{\rho(I)}$ and $p_{\rho(I)}$, respectively.
Second, every symmetric function 
$\sum_{\lambda\vdash n}c_\lambda e_\lambda$
has an infinite number of noncommutative analogs 
$\sum_{I\vDash n}c_I'\Lambda^I$
in $\mathrm{NSym}$,
in which only a finite number are $\Lambda$-positive
with integer coefficients.
Third,
a symmetric function is $e$-positive 
if and only if it has a $\Lambda$-positive noncommutative analog.
For the purpose of
showing the $e$-positivity of a chromatic symmetric function $X_G$,
one may follow the steps below.
\begin{description}
\item[Step 1]
Initiate the argument by 
deriving a noncommutative analog $\widetilde X_G$ in its $\Lambda$-expansion.
We know two ways to achieve this. 
One is to start from the $p$-expansion of~$X_G$ by definition,
which implies 
the $\Psi$-expansion of a noncommutative analog directly.
Then we transform the analog
to its $\Lambda$-expansion by change-of-basis,
see \cref{sec:appendix} for this approach working for cycles.
The other way is to compute~$X_G$ by applying 
\citeauthor{OS14}'s \emph{triple-deletion property} \cite{OS14},
and by using graphs with known $e_I$-expansions,
see \cref{thm:tadpole} for this way working for tadpoles.
\item[Step 2]
Find a positive $e_I$-expansion.
Decompose the set of all compositions of $n=\abs{V(G)}$
as $\mathcal I^{(1)}\sqcup\dotsm\mathcal \sqcup I^{(l)}$, such that 
\begin{enumerate}
\item
$\widetilde X_G
=\sum_{k=1}^l
\sum_{I\in\mathcal I^{(k)}}
c_I\Lambda^I$,
\item
the compositions in each $\mathcal I^{(k)}$ 
have the same underlying partition, say, $\lambda^{(k)}$,
and
\item
the inner sum for each $k$ 
has an $e$-positive commutative image, i.e., 
$\sum_{I\in\mathcal I^{(k)}}c_I\ge 0$.
\end{enumerate}
It follows that 
\begin{equation}\label{cm}
X_G
=\sum_{k=1}^l
\brk4{
\sum_{I\in\mathcal I^{(k)}}
c_I
}
e_{\lambda^{(k)}}
\end{equation}
is a positive $e_I$-expansion.
\item[Step 3]
Produce a neat $e_I$-expansion by shaping \cref{cm}.
One thing we can do is to simplify each of the coefficients 
$\sum_{I\in\mathcal I^{(k)}}c_I$ for given composition functions $c_I$.
Another thing is to further merge the terms 
for distinct indices, say $k$ and $h$, 
with the same underlying partition $\lambda^{(k)}=\lambda^{(h)}$.
Sign-reversing involutions, injections and bijections
may help embellish expressions to make them compact and elegant.
\end{description}

One may catch a whiff of the combinatorial essence of 
the composition method from each of the steps.
Besides suitably selecting a vertex triple to apply the triple-deletion
property,
a vast flexibility lies 
in both the process of decomposing and coefficient shaping.
We wish that the $e$-positivity of \cref{cm}
is as transparent as the $e$-positivity in \cref{X.path}.
Step $3$ is not necessary 
for the sole purpose of positivity establishment,
however, 
it would be computationally convenient if 
we make use of a neat $e_I$-expansion 
in proving the $e$-positivity of graphs that are of more complex.

In this paper,
we start the journey of understanding
the computing power of
the composition method
in proving the $e$-positivity of graphs.
%It is coming up roses in 2024 to produce neat formulas 
%for chromatic symmetric functions,
%and to show their $e$-positivity as byproducts.
%See \citet{CHW24X,TW24X,QTW24X,TWW24X,Wang24X}.

After making necessary preparations in \cref{sec:preliminary},
we apply the composition method for special families of graphs in \cref{sec:formulas}.
We work out neat formulas for tadpoles and barbells.
The former are particular squids that were investigated by \citet{MMW08}, see also \citet{LLWY21},
while the latter contains lollipops, lariats and dumbbells as specializations.
Using the composition method,
we also establish the $e$-positivity of hats.
The family of hats 
contains both tadpoles and generalized bulls.
Our result for hats induces a second $e_I$-expansion
for tadpoles.
The family of generalized bulls
was listed as an infinite collection of $e$-positive claw-free graphs 
that are not claw-contractible-free by \citet[Section~$3$]{DFv20}.
We also consider the line graphs of tadpoles,
since the line graph of any graph is claw-free,
which is a key condition in both \cref{conj:epos,conj:spos}.

An early try of the composition method
towards Schur positivity is 
\cite{TW23X}, in which \citeauthor{TW23X}
obtained the ribbon Schur expansion
of a noncommutative analog
for spiders of the form $S(a,2,1)$.
They are not ribbon positive. 
This analog yields a skew Schur expansion of~$X_{S(a,2,1)}$.
By the Littlewood--Richardson rule,
the ordinary Schur coefficients are by that means 
multiset sizes of Yamanouchi words,
and the Schur positivity then follows by injections.
A similar proof for the Schur positivity of spiders of the form $S(a,4,1)$
is beyond uncomplicated. 
We thereby expect more satisfying applications of the composition method
in establishing the Schur positivity of graphs.
In this paper, 
we give a compact ribbon Schur analog for the chromatic symmetric function of cycles, see \cref{thm:ribbon.Schur:cycle}.

\section{Preliminaries}\label[sec]{sec:preliminary}

This section contains necessary notion and notation, 
basic results on commutative symmetric functions, 
chromatic symmetric functions, and noncommutative symmetric functions,
that will be of use.

\subsection{Compositions and partitions}
We use terminology from \citet{Sta11B}.
Let $n$ be a positive integer. 
A \emph{composition} of $n$ is 
a sequence of positive integers with sum~$n$,
commonly denoted 
$I
=i_1 \dotsm i_s
\vDash n$.
It has \emph{size} $\abs{I}=n$,
\emph{length} $\ell(I)=s$, and \emph{reversal}
$\overline{I}
=i_s i_{s-1}\dotsm i_1$.
The integers $i_k$ are called \emph{parts} of~$I$.
For notational convenience, 
we write $I=v^s$ if all parts have the same value $v$,
and denote the $k$th last part as $i_{-k}$;
thus $i_{-1}=i_s$.
We consider the number~$0$ to have a unique composition, 
denoted~$\epsilon$.
Whenever a capital letter such like $I$ and $J$
is adopted to denote a composition,
we use the small letter counterparts such as $i$ and $j$ respectively
with integer subscripts to denote the parts.
A \emph{factor} of $I$
is a subsequence that consists of consecutive parts.
A \emph{prefix} (resp., \emph{suffix}) of $I$
is a factor that starts from~$i_1$ (resp., ends at~$i_s$).
Denote by $m_k(I)$ 
the the number of parts $k$ in $I$, namely,
\begin{equation}\label[def]{def:mk}
m_k(I)
=\abs{\{j\in\{1,\dots,s\}\colon i_j=k\}}.
\end{equation}

A \emph{partition}\index{partition} of $n$
is a multiset of positive integers with sum $n$,
commonly denoted as
\[
\lambda=\lambda_1\lambda_2\dotsm
=1^{m_1(\lambda)}2^{m_2(\lambda)}\dotsm
\vdash n,
\]
where $\lambda_1\ge \lambda_2\ge\dotsm\ge 1$.
For any composition $I$, 
there is a unique partition $\rho(I)$ satisfying \cref{def:rho},
i.e., the partition obtained by rearranging the parts of $I$.
As partitions have \emph{Young diagrams} as graphic representation,
one uses the terminology \emph{ribbons} to illustrate compositions.
In French notation,
the ribbon for a composition~$I$
is the collection of boxes such that
\begin{itemize}
\item
Row $k$ consists of $i_k$ consecutive boxes, and
\item
the last box on Row $k$ and the first box on Row $k+1$ 
are in the same column.
\end{itemize}
In the theory of integer partitions, 
by saying a Young diagram $\lambda$ 
one emphasizes the geometric shape of the partition $\lambda$.
Being analogous in our composition calculus,
we phrase the wording ``a ribbon~$I$''
to call attention to the illustration of the composition $I$.

Following \citet{Mac1915B},
the \emph{conjugate} $I^\sim$ of a composition~$I$
is the ribbon consisting of the column lengths 
of~$I$ from right to left. 
This is different to the \emph{conjugate} 
$\lambda'$ of a partition $\lambda$,
whose Young diagram is obtained by turning rows into columns.
For example,
$32^\sim=121^2$
and
$32'=221$.
A \emph{refinement} of $I$ is a composition 
$J=j_1\dotsm j_t$ such that
\[
i_1
=
j_{k_{0}+1}+\dots+j_{k_1},\quad
\dots,\quad
i_s
=
j_{k_{s-1}+1}+\dots+j_{k_s},
\]
for some integers $k_0<\dots<k_s$,
where $k_0=0$ and $k_s=t$.
We say that $I$ is a \emph{coarsement} of $J$ 
if $J$ is a refinement of $I$.
The \emph{reverse refinement order} $\preceq$ for compositions 
is the partial order defined by
\[
I\preceq J
\iff
\text{$J$ is a refinement of $I$}.
\]
The \emph{first parts of blocks of~$J$ with respect to $I$}
are the numbers $j_{k_0+1},\dots,j_{k_{s-1}+1}$,
with product 
\[
f\!p(J,I)
=
j_{k_0+1}\dotsm j_{k_{s-1}+1}.
\]
The \emph{last parts of blocks of~$J$ with respect to $I$}
are the numbers $j_{k_1},\dots,j_{k_s}$,
with product 
\[
lp(J,I)
=
j_{k_1}\dotsm j_{k_s}.
\]
By definition, one may derive directly that 
\begin{equation}\label{lp=fp}
lp(\overline{J},\overline{I})
=
f\!p(J,I).
\end{equation}
For any compositions 
$I=i_1\dotsm i_s$
and
$J=j_1\dotsm j_t$,
the \emph{concatenation} 
of $I$ and $J$ is the composition
$I\!J
=
i_1
\dotsm
i_s
j_1
\dotsm
j_t$,
and the \emph{near concatenation} of $I$ and $J$ is the composition
\[
I\triangleright J
=
i_1
\dotsm
i_{s-1}
(i_s+j_1)
j_2
\dotsm
j_t.
\]
In French notation,
the ribbon $I\!J$ (resp., $I\triangleright J$)
is obtained by attaching the first box of~$J$
immediately below (resp., to the immediate right of)
the last box of $I$.

The \emph{decomposition of a ribbon $J$
relatively to a composition $I$}
is the unique expression 
\[
\nabla_I(J)
=
J_1\bullet_1
J_2\bullet_2
\dots
\bullet_{s-1} J_s,
\]
where $s=\ell(I)$,
each $J_k$ is a ribbon of size $i_k$, 
and each symbol $\bullet_k$ stands for
either the concatenation
or the near concatenation.
For instance,
\[
\nabla_{83}(5141)
=512\triangleright 21.
\]
We call the ribbons $J_k$ \emph{blocks} of $\nabla_I(J)$.
In the language of ribbons,
the block $J_k$
consists of the first~$i_k$ boxes 
of the ribbon that is obtained from~$J$
by removing the previous blocks $J_1,\dots,J_{k-1}$.

A \emph{hook} is a ribbon of the form $1^s t$ 
for some $s\ge 0$ and $t\ge 1$.
Every hook appears as the English letter~L or a degenerate one, 
that is, a horizontal ribbon $t$
or a vertical ribbon~$1^s$. 
Here we recognize the ribbon $1$ as horizontal.
Denote by~$\mathcal H_I$
the set of ribbons $J$ such that every block 
in the decomposition~$\nabla_I(J)$ is a hook. Then
\[
\mathcal H_n
=\{n,\
1(n-1),\
1^2(n-2),\
\dots,\
1^{n-2}2,\
1^n\}
\]
is the set of hooks of the composition~$n$ consisting of a single part. 
Moreover,
since every factor of a hook is still a hook,
we have $\mathcal H_n
\subseteq\mathcal H_I$
for all $I\vDash n$.
For example,
$\mathcal H_4
=\{4,\
13,\
1^22,\
1^4\}$,
$\mathcal H_{31}
=\mathcal H_4\cup\{31,\
121\}$,
and
$\mathcal H_{13}
=\mathcal H_4\cup\{22,\
21^2\}$.
Let $I=i_1\dotsm i_s$.
By definition,
the set $\mathcal H_I$ 
is in a bijection with the set
\[
\{
J_1\bullet_1
J_2\bullet_2
\dots
\bullet_{s-1} J_s
\colon
J_k\in\mathcal H_{i_k}\text{ for $1\le k\le s$, and }
\bullet_k\in\{\triangleleft,\,\triangleright\}
\text{ for $1\le k\le s-1$}
\},
\]
where the symbol $\triangleleft$ stands for the concatenation operation.
As a consequence, one may calculate
$\abs{\mathcal H_I}=2^{s-1}i_1\dotsm i_s$.

\subsection{Commutative symmetric functions}\label[sec]{sec:sf}
We give an overview of necessary notion and notation for 
the theory of commutative symmetric functions. 
For comprehensive references,
one may refer to \citet{Sta23B} and \citet{MR15B}.
Let $R$ be a commutative ring with identity. 
A \emph{symmetric function} of homogeneous degree $n$ 
over~$R$ is a formal power series
\[
f(x_1, x_2, \dots)
=
\sum_{\lambda=\lambda_1 \lambda_2 \dotsm \vdash n}
c_\lambda \cdot x_1^{\lambda_1} x_2^{\lambda_2} \dotsm, 
\quad \text { where } c_\lambda \in R,
\]
such that 
$f(x_1, x_2, \dots)
=f(x_{\pi(1)}, x_{\pi(2)}, \dots)$ for any permutation $\pi$.
Denote by $\mathbb Q$ the field of rational numbers. 
Define $\operatorname{Sym}^0=\mathbb Q$, 
and define $\operatorname{Sym}^n$ 
to be the vector space of homogeneous symmetric functions 
of degree $n$ over $\mathbb Q$. 
Common bases of $\mathrm{Sym}^n$ include
the \emph{elementary symmetric functions}~$e_\lambda$,
the \emph{complete homogeneous symmetric functions}~$h_\lambda$,
the \emph{power sum symmetric functions}~$p_\lambda$,
and \emph{the Schur symmetric functions}~$s_\lambda$.
The first three ones are multiplicatively defined by
\[
b_\lambda
=b_{\lambda_1}\dotsm b_{\lambda_l},
\quad\text{for $b\in\{e,h,p\}$ 
and for any partition 
$\lambda=\lambda_1\dotsm\lambda_l$},
\]
where
\[
e_k
=\sum_{1\le i_1<\dots<i_k} 
x_{i_1} \dotsm x_{i_k},
\quad
h_k
=\sum_{1\le i_1\le \dots \le i_k}
x_{i_1} \dots x_{i_k},
\quad\text{and}\quad
p_k
=\sum_{i\ge 1} x_i^k.
\]
The Schur symmetric function $s_\lambda$ 
can be defined combinatorially by
$s_\lambda
=\sum_{T\in\mathrm{CS}_\lambda}w(T)$,
where $\mathrm{CS}_\lambda$ 
is the set of column strict tableaux of shape $\lambda$,
and the weight $w(T)$ is the product of $x_i$ for all entries~$i$ in $T$.
Here a tableau of shape $\lambda$ is said to be \emph{column strict} if
\begin{itemize}
\item
the entries in each row weakly increase, and
\item
the entries in each column strictly increase 
starting from the longest row; 
this is to say from bottom to top in French notation.
\end{itemize}
The Schur symmetric functions are said to be
``the most important basis for $\mathrm{Sym}$ 
with respect to its relationship 
to other areas of mathematics'' 
and ``crucial in understanding the representation theory 
of the symmetric group,''
see \cite[Page 37]{MR15B}.

For any basis $\{b_\lambda\}$ of $\mathrm{Sym}^n$ and
any symmetric function $f\in \mathrm{Sym}^n$,
the \emph{$b_\lambda$-coefficient} of $f$
is the unique number $c_\lambda$ such that 
$f=\sum_{\lambda\vdash n}c_\lambda b_\lambda$,
denoted $[b_\lambda]f=c_\lambda$.
The symmetric function $f$ is said to be \emph{$b$-positive}
if every $b$-coefficient of $f$ is nonnegative.
For instance, 
every elementary symmetric function is Schur positive since 
$
e_\lambda
=
\sum_{\mu\vdash\abs{\lambda}} 
K_{\mu'\lambda}
s_\mu$,
where $K_{\mu'\lambda}$ are Kostka numbers,
see~\cite[Exercise~2.12]{MR15B}.

With the aid of the function $\rho$ defined by \cref{def:rho},
one may extend the domain of these basis symmetric functions 
from partitions to compositions.
Precisely speaking, 
one may define
$b_I=b_{\rho(I)}$
for any composition~$I$ 
and any basis $\{b_\lambda\}_\lambda$.
With this convention, 
we are safe to write $e_I$ instead 
of the redundant expression $e_{\rho(I)}$.
Since $\{e_I\}_{I\vDash n}$ is not a basis of $\mathrm{Sym}^n$,
the notation $[e_I]f$ is undefined.

\subsection{Chromatic symmetric functions}\label[sec]{sec:csf}

\citet{Sta95} introduced the \emph{chromatic symmetric function}
for a graph $G$ as
\[
X_G=\sum_\kappa \prod_{v \in V(G)} \mathbf{x}_{\kappa(v)},
\]
where $\mathbf{x}=\left(x_1, x_2, \ldots\right)$ 
is a countable list of indeterminates, 
and $\kappa$ runs over proper colorings of~$G$. 
Chromatic symmetric functions are particular symmetric functions,
and it is a generalization of 
Birkhoff's chromatic polynomials $\chi_G(k)$, since
$X_G(1^k00\dotsm)=\chi_G(k)$.
For instance, the chromatic symmetric function 
of the complete graph $K_n$ is 
\begin{equation}\label{X.complete}
X_{K_n}=n!e_n.
\end{equation}

We will need the $p$-expansion of $X_G$,
see \cite[Theorem 2.5]{Sta95}.

\begin{proposition}[\citeauthor{Sta95}]\label{prop:csf.p}
The chromatic symmetric function of a graph $G=(V,E)$ is
\[
X_G
=\sum_{E'\subseteq E}(-1)^{\abs{E'}}p_{\tau(E')}
\]
where $\tau(E')$ 
is the partition consisting of
the component orders of the spanning subgraph~$(V,E')$.
\end{proposition}

By \cite[Theorem 2.22]{MR15B}, 
every $e$-coefficient in a power sum symmetric function~$p_\mu$
is an integer.
It then follows from \cref{prop:csf.p} that every $e$-coefficient of $X_G$ is integral.
\citet[Corollary 3.6]{Sta95} presented 
the following quick criterion for the $e$-positivity.

\begin{proposition}[\citeauthor{Sta95}]\label{prop:epos:alpha<=2}
Any graph whose vertices can be partitioned into 
two cliques is $e$-positive.
\end{proposition}

Such graphs have several characterizations, such as
the complements of bipartite graphs 
and the incomparability graphs of $3$-free posets,
see~\citet[Theorem 5.3]{Gua13X}. 
\citet[Propositions~5.3 and 5.4]{Sta95} confirmed the $e$-positivity of paths and cycles.

\begin{proposition}[\citeauthor{Sta95}]\label{prop:gf:path+cycle}
Let $E(z)=\sum_{n\ge0} e_n z^n$ and $F(z)=E(z)-zE'(z)$.
Denote by $P_n$ the $n$-vertex path and by $C_n$ the $n$-vertex cycle.
Then
\[
\sum_{n\geq 0}X_{P_n}z^n
=
\frac{E(z)}{F(z)}
\quad\text{and}\quad
\sum_{n\geq 2}X_{C_n}z^n
=
\frac{z^2E''(z)}{F(z)}.
\]
As a consequence, paths and cycles are $e$-positive.
\end{proposition}

Explicit formulas for the $e$-coefficients of $X_{P_n}$ and $X_{C_n}$
were obtained by extracting the coefficients of these generating functions,
see \citet[Theorem~3.2]{Wol98}.
\citet{SW16} obtained the much simpler \cref{X.path} for paths.
\citet[Corollary~6.2]{Ell17} gave a formula
for the chromatic quasisymmetric function of cycles,
whose $t=1$ specialization is an
equally simple one.

\begin{proposition}[\citeauthor{Ell17}]\label{prop:cycle}
For $n\ge 2$,
$X_{C_n}
=\sum_{I\vDash n}
(i_1-1)
w_I
e_I$.
\end{proposition}
We provide a proof for \cref{prop:cycle}
using the composition method in \cref{sec:appendix}.
\Citet[Theorem 3.1, Corollaries 3.2 and 3.3]{OS14} established the triple-deletion property for chromatic symmetric functions.

\begin{theorem}[\citeauthor{OS14}]\label{thm:3del}
Let $G$ be a graph with a stable set $T$ of order $3$.
Denote by $e_1$, $e_2$ and $e_3$ the edges linking the vertices in $T$.
For any set $S\subseteq \{1,2,3\}$, 
denote by $G_S$ the graph with vertex set~$V(G)$
and edge set $E(G)\cup\{e_j\colon j\in S\}$.
Then 
\[
X_{G_{12}}
=X_{G_1}+X_{G_{23}}-X_{G_3}
\quad\text{and}\quad
X_{G_{123}}
=X_{G_{13}}+X_{G_{23}}-X_{G_3}.
\]
\end{theorem}

\subsection{Noncommutative symmetric functions}\label[sec]{sec:NSym}
For an introduction and basic knowledge on noncommutative symmetric functions,
see \citet{GKLLRT95}.
Let~$K$ be a field of characteristic zero.
The algebra of \emph{noncommutative symmetric functions}
is the free associative algebra 
$\mathrm{NSym}
=K\langle\Lambda_1,\Lambda_2,\dots\rangle$
generated by an infinite sequence $\{\Lambda_k\}_{k\ge 1}$ of indeterminates
over $K$, where $\Lambda_0=1$.
It is graded by the weight function $w(\Lambda_k)=k$.
The homogeneous component of weight $n$ is denoted $\mathrm{NSym}_n$.
Let $t$ be an indeterminate that commutes with all indeterminates $\Lambda_k$.
The \emph{elementary symmetric functions} are $\Lambda_n$ themselves,
whose generating function is denoted by
\[
\lambda(t)
=
\sum_{n\ge 0}\Lambda_n t^n.
\]
The \emph{complete homogeneous symmetric functions} $S_n$ are defined by the generating function
\[
\sigma(t)
=\sum_{n\ge 0} S_n t^n
=\frac{1}{\lambda(-t)}.
\]
The \emph{power sum symmetric functions $\Psi_n$ of the first kind} 
are defined by the generating function
\[
\psi(t)
=\sum_{n\ge 1}\Psi_n t^{n-1}
=\lambda(-t)\sigma'(t).
\]
For any composition $I=i_1 i_2\dotsm$, define
\[
\Lambda^I
=
\Lambda_{i_1}
\Lambda_{i_2}\dotsm,
\quad
S^I
=
S_{i_1}
S_{i_2}\dotsm,
\quad\text{and}\quad
\Psi^I
=
\Psi_{i_1}
\Psi_{i_2}\dotsm.
\]
The algebra $\mathrm{NSym}$ is freely generated by any one 
of these families.
Here the superscript notation are adopted to indicate that
the functions are multiplicative 
with respect to composition concatenations.
The \emph{sign} of $I$ is defined by
\begin{equation}\label[def]{def:sign.I}
\varepsilon^I
=
(-1)^{\abs{I}-\ell(I)}.
\end{equation}
It is direct to check that 
\begin{equation}\label{epsilon:IJ}
\varepsilon^I\varepsilon^J
=\varepsilon^{I\!J}.
\end{equation}
Another linear basis of $\mathrm{NSym}$ is 
the \emph{ribbon Schur functions} $R_I$, which can be defined by
\[
\varepsilon^I R_I
=
\sum_{J\preceq I}
\varepsilon^J S^J,
\]
see \cite[Formula (62)]{GKLLRT95}.
We list some transition rules for these bases,
see \cite[Propositions 4.15 and~4.23, and Note 4.21]{GKLLRT95}.

\begin{proposition}[\citeauthor{GKLLRT95}]\label{prop:transition:NSym}
For any composition $I$, we have 
\begin{align}
\label{Lambda2Ribbon}
\Lambda^I
&=
\sum_{J\succeq \overline{I}^\sim}
R_J,\\
\label{Psi2Lambda}
\Psi^I
&=
\sum_{J\succeq I}
\varepsilon^J 
f\!p(J,I)
\Lambda^J,
\quad\text{and}\\
\label{Psi2Ribbon}
\Psi^I
&=
\smashoperator[r]{
\sum_{J\in\mathcal H_I}}
\varepsilon^{I\!J_1\dotsm J_{\ell(I)}}
R_J,
\end{align}
where $J_k$ are the composition blocks
of the decomposition $\nabla_I(J)$.
\end{proposition}

\Cref{Psi2Lambda} is true by virtue of \cref{lp=fp},
though it was expressed in terms of the
product $lp(J,I)$ in \cite{GKLLRT95}.
Recall from \cref{def:rho}
that $\rho$ 
maps a composition to its underlying partition.
We use the same notation $\rho$ to denote 
the projection map
defined by $\rho(\Lambda^I)=e_I$
and by extending it linearly.
By definition, for any composition~$I$,
\[
\rho(\Lambda^I)
=e_I,
\quad
\rho(S^I)
=h_I,
\quad
\rho(\Psi^I)
=p_I,
\quad\text{and}\quad
\rho(R_I)
=s_{\mathrm{sh}(I)},
\]
where $\mathrm{sh}(I)$ is
the skew partition of shape $I$.
For instance, 
\[
\rho(\Lambda^{12})
=
e_{21},
\quad
\rho(S^{12})
=
h_{21},
\quad
\rho(\Psi^{12})
=
p_{21},
\quad
\rho(R_{12})
=
s_{21}
\quad\text{and}\quad
\rho(R_{21})
=
s_{22/1}.
\]
When $\rho(F)=f$ for some $F\in\mathrm{NSym}$ and $f\in\mathrm{Sym}$,
we say that $f$ is the \emph{commutative image} of $F$,
and that $F$ is a \emph{noncommutative analog} of~$f$.
For instance, \cref{X.path,prop:cycle} imply that 
$X_{P_n}$ and $X_{C_n}$ have the noncommutative analogs
\begin{align}
\label{wX.path}
\widetilde X_{P_n}
&=
\sum_{I\vDash n}
w_I \Lambda^I,
\quad\text{and}\\
\label{wX.cycle}
\widetilde X_{C_n}
&=
\sum_{I\vDash n}
(i_1-1)
w_I
\Lambda^I,
\end{align}
respectively.
If a chromatic symmetric function $X_G$ has
a noncommutative analog $\widetilde X_G\in\mathrm{NSym}$,
then for any partition $\lambda\vdash \abs{V(G)}$,
\[
[e_\lambda]X_G
=\sum_{\rho(I)=\lambda}[\Lambda^I]\widetilde X_G.
\] 
The aforementioned ambiguity issue 
is solved naturally 
in the language of the algebra $\mathrm{NSym}$.
Indeed, 
since $\{\Lambda^I\}_{I\vDash n}$ is a basis of $\mathrm{NSym}_n$,
we talk about the well defined $\Lambda^I$-coefficients
instead of the undefined ``$e_I$-coefficients''.

By definition, any chromatic symmetric function 
has an infinite number of noncommutative analogs,
among which only a finite number with integer coefficients 
are $e$-positive.
In particular, if a symmetric function
$\sum_{\lambda\vdash n}c_\lambda e_\lambda$ 
is $e$-positive, then the analog 
$\sum_{\lambda\vdash n}c_\lambda \Lambda^\lambda$
is $\Lambda$-positive.
Therefore, a symmetric function is $e$-positive 
if and only if it has a $\Lambda$-positive 
noncommutative analog.
Therefore, in order to prove that a graph $G$ is $e$-positive,
it suffices to find a $\Lambda$-positive analog of $X_G$.
The algebra $\mathrm{NSym}$ plays the role of
providing theoretical support for the composition method.
As a consequence, 
we display only positive $e_I$-expansions in theorem statements.
We would not write in terms of noncommutative analogs 
except when arguing $\Lambda^I$-coefficients 
is convenient.

\subsection{Warming up for the composition method}\label[sec]{sec:CompFn}

This section consists of a property of the function~$w_I$ defined by \cref{def:w},
some other composition functions and their interrelations,
as well as some practices of using these functions.

From definition, 
it is straightforward to see that $w_I=w_J$
for any composition $J$ that is 
obtained by rearranging the non-first parts of $I$.
Another five-finger exercise is as follows.

\begin{lemma}\label{lem:wIJ}
Let $I$ and $J$ be nonempty compositions such that $j_1\ne 1$.
Then
\[
w_I w_J
=
\frac{j_1}{j_1-1}\cdotp 
w_{K}
\]
for any composition $K$ that is obtained by rearranging the parts of $I\!J$ 
such that $k_1=i_1$.
\end{lemma}
\begin{proof}
Direct by \cref{def:w}.
%Let $K$ be a composition obtained by rearranging all parts of $I\!J$
%such that $k_1=i_1$.
%Then
%\[
%(j_1-1) w_I w_J
%=
%(j_1-1) i_1
%\prod_{k\ge 2}(i_k-1)
%\cdotp j_1\prod_{h\ge 2}(j_h-1)
%=
%j_1 w_K.
%\]
%For $j_1\ne 1$, dividing both sides by $j_1-1$ yields the desired formula.
\end{proof}

For any number $a\le \abs{I}$, 
we define
the \emph{surplus partial sum} of $I$ with respect to $a$
to be the number
\begin{equation}\label[def]{def:sigma+}
\sigma_I^+(a)
=
\min\brk[c]1{\abs{i_1\dotsm i_k}\colon
0\le k\le \ell(I),\
\abs{i_1\dotsm i_k}\ge a}.
\end{equation}
Define the \emph{$a$-surplus} of $I$ to be the number
\begin{equation}\label[def]{def:Theta+}
\Theta_I^+(a)
=
\sigma_I^+(a)-a.
\end{equation}
Then $\Theta_I^+(a)\ge 0$.
The function $\Theta_I^+(\cdot)$ will appear in \cref{thm:tadpole}.
Here is a basic property.

\begin{lemma}\label{lem:Theta+.t}
Let $I\vDash n$ and $0\le a,t\le n$.
If $\Theta_I^+(a)\ge t$, then 
$
\Theta_I^+(a)
=
t+\Theta_I^+(a+t)$.
\end{lemma}
\begin{proof}
This is transparent if one notices 
$\sigma_I^+(a+t)=\sigma_I^+(a)$.
\end{proof}

\Cref{lem:Theta+.t} will be used in the proof of \cref{thm:hat}. 
Similarly, for any number $a\ge 0$, 
we define
the \emph{deficiency partial sum} of $I$ with respect to $a$ 
to be the number
\begin{equation}\label[def]{def:sigma-}
\sigma_I^-(a)
=
\max\brk[c]1{
\abs{i_1\dotsm i_k}
\colon
0\le k\le \ell(I),\
\abs{i_1\dotsm i_k}\le a
},
\end{equation}
and define the \emph{$a$-deficiency} of $I$ to be the number
\begin{equation}\label[def]{def:Theta-}
\Theta_I^-(a)
=
a-\sigma_I^-(a).
\end{equation}
Then $\Theta_I^-(a)\ge 0$.
The function $\sigma^-_I$ (resp., $\Theta^-_I$)
can be expressed in terms of $\sigma^+_I$ (resp., $\Theta^+_I$).

\begin{lemma}\label{lem:sigma+-:Theta+-}
Let $I\vDash n$ and $0\le a\le n$. Then
\begin{equation}
\label{sigma-:sigma+}
\sigma_I^-(a)
=
n
-
\sigma_{\overline{I}}^+(n-a),
\end{equation}
or equivalently,
\begin{equation}
\label{Theta-:Theta+}
\Theta_I^-(a)
=
\Theta_{\overline{I}}^+(n-a).
\end{equation}
\end{lemma}
\begin{proof}
We shall show \cref{sigma-:sigma+} first.
If $a=n$, then 
$\sigma_I^-(a)=n$
and
$\sigma_{\overline{I}}^+(n-a)=0$,
satisfying \cref{sigma-:sigma+}.
Suppose that $0\le a<n$,
and 
\begin{equation}\label{pf:sigma-}
\sigma_I^-(a)=\abs{i_1\dotsm i_k}.
\end{equation}
Then $0\le k\le \ell(I)-1$. By \cref{def:sigma-},
\[
\abs{i_1\dotsm i_k}
\le a
<\abs{i_1\dotsm i_{k+1}}.
\]
Subtracting from $n$ by each sum in the above inequality, 
we obtain
\[
\abs{i_{k+1}\dotsm i_{-1}}
\ge n-a
>
\abs{i_{k+2}\dotsm i_{-1}},
\]
which reads,
$
\sigma_{\overline{I}}^+(n-a)
=
\abs{i_{k+1}\dotsm i_{-1}}$.
Adding it up with \cref{pf:sigma-}, 
we obtain the sum $n$ as desired.
This proves \cref{sigma-:sigma+}.
Using \cref{def:Theta+,def:Theta-}, 
one may infer \cref{Theta-:Theta+} from \cref{sigma-:sigma+}.
This completes the proof.
\end{proof}

\cref{lem:sigma+-:Theta+-} will be used in the proof of \cref{thm:epos:hat}.
Let us express the product $X_{P_l}X_{C_m}$ in terms of the functions $w_I$ and $\Theta_I^+(\cdot)$.

\begin{lemma}\label{lem:path*cycle}
For $l\ge 1$ and $m\ge 2$,
\begin{align}
\label{fml:PathCycle.j1}
X_{P_l}
X_{C_m}
&=
\sum_{
I\vDash l,\,
J\vDash m}
j_1
w_{I\!J}
e_{I\!J}\\
\label{fml:PathCycle.K}
&=
\sum_{
K\vDash l+m,\
\Theta_K^+(l)=0
}
\brk1{\Theta_K^+(l+1)+1}
w_K
e_K.
\end{align}
\end{lemma}
\begin{proof}
By \cref{X.path,prop:cycle,lem:wIJ},
\[
X_{P_l}
X_{C_m}
=
\sum_{I\vDash l}
w_I 
e_I
\sum_{J\vDash m}
(j_1-1)
w_J 
e_J
=
\sum_{
I\vDash l,\
J\vDash m
}
j_1
w_{I\!J}
e_{I\!J}.
\]
This proves \cref{fml:PathCycle.j1}.
The other formula holds since $j_1=\Theta_K^+(l+1)+1$
when $K=I\!J$.
\end{proof}

Note that neither of \cref{fml:PathCycle.j1,fml:PathCycle.K} holds for $l=0$.
Now we compute a partial convolution of $X_{P_l}$ and $X_{C_m}$.

\begin{lemma}\label{lem:convolution:path*cycle}
For $0\le l\le n-2$,
\begin{equation}\label{fml:convolution:path*cycle}
\sum_{k=0}^l 
X_{P_k}
X_{C_{n-k}}
=\sum_{I\vDash n}
\brk1{\sigma_I^+(l+1)-1}
w_I
e_I.
\end{equation}
Dually, for $2\le m\le n-1$,
\begin{equation}\label{fml:convolution:cycle*path}
\sum_{i=2}^m
X_{C_i}
X_{P_{n-i}}
=
\sum_{I\vDash n}
\sigma_{\overline{I}}^-(m)
w_I
e_I.
\end{equation}
\end{lemma}
\begin{proof}
By \cref{fml:PathCycle.j1},
the convolution on the left hand of \cref{fml:convolution:path*cycle} has a noncommutative analog
\[
\sum_{k=1}^l
\widetilde X_{P_k}
\widetilde X_{C_{n-k}}
=
\sum_{k=1}^l 
\sum_{
I\vDash k,\
J\vDash n-k}
j_1 w_{I\!J}\Lambda^{I\!J}
=
\sum_{
K=I\!J\vDash n,\
1\le \abs{I}\le l
}
j_1
w_K
\Lambda^K.
\]
Combining it with \cref{prop:cycle}, we obtain
\begin{align*}
\sum_{k=0}^l
\widetilde X_{P_k}
\widetilde X_{C_{n-k}}
&=
\sum_{
K=I\!J\vDash n,\
0\le \abs{I}\le l
}
j_1
w_K
\Lambda^K
-
\sum_{K\vDash n}
w_K
\Lambda^K.
\end{align*}
The coefficient of $w_K\Lambda^K$
of the first sum on the right side is the partial sum $k_1+\dots+k_r$ 
such that 
\[
\abs{k_1\dotsm k_{r-1}}
\le l
<\abs{k_1\dotsm k_r},
\]
that is, the sum $\sigma_K^+(l+1)$.
This proves \cref{fml:convolution:path*cycle}.
In the same fashion, one may show \cref{fml:convolution:cycle*path}.
\end{proof}

We need the noncommutative setting in the proof above
since the coefficient of $w_K\Lambda^K$
is considered.
\Cref{lem:convolution:path*cycle}
will be used in the proof of \cref{thm:tadpole} for tadpoles.

\begin{corollary}\label{cor:convolution.full:path*cycle}
For $n\ge 2$,
the average 
of the full convolution of chromatic symmetric functions of paths and cycles
with total order $n$ is 
the chromatic symmetric function of the path of order $n$, i.e.,
\[
\frac{1}{n-1}\sum_{k=0}^{n-2}
X_{P_k}
X_{C_{n-k}}
=
X_{P_n}.
\]
\end{corollary}
\begin{proof}
Taking $l=n-2$ in \cref{fml:convolution:path*cycle}, and using \cref{X.path},
one obtains
the desired formula.
\end{proof}

It can be shown alternatively by taking $m=n-1$ in \cref{fml:convolution:cycle*path}
and using \cref{prop:cycle}
and the identity $\Theta_{\overline{I}}^-(n-1)=n-i_1$,
or, by \cref{prop:gf:path+cycle}.

\section{Neat formulas for some chromatic symmetric functions}\label[sec]{sec:formulas}

In this section,
we use the composition method to produce
neat formulas for the chromatic symmetric functions
of several families of graphs, 
including tadpoles and their line graphs, barbells,
and generalized bulls.
We also establish the $e$-positivity of hats.

\subsection{The ribbon expansion for cycles}\label[sec]{sec:cycle}

In view of \cref{Lambda2Ribbon},
if a noncommutative symmetric function~$F$
is $\Lambda$-positive, then it is $R$-positive.
\citet{TW23X} discovered that the analog $\widetilde X_{P_n}$
has the rather simple ribbon expansion 
\[
\widetilde X_{P_n}
=
\sum_{
I\vDash n,\
i_{-1}=1,\
i_1,\dots,i_{-2}\le 2
}
2^{m_1(I)-1}
R_I.
\]
We present a $\Psi$-expansion for a noncommutative analog of cycles.

\begin{lemma}\label{lem:Psi:cycle}
For $n\ge 2$,
the chromatic symmetric function $X_{C_n}$ has a noncommutative analog
\[
\widetilde X_{C_n}
=
(-1)^n
\Psi^n
+
\sum_{I\vDash n}
\varepsilon^I
i_1
\Psi^I,
\]
where $\varepsilon^I$ is defined by \cref{def:sign.I}.
\end{lemma}
\begin{proof}
Let $C_n=(V,E)$ be the cycle with vertices $v_1,\dots,v_n$ 
arranged counterclockwise.
Let $E'\subseteq E$. 
The contribution of the edge set 
$E'=E$ in \cref{prop:csf.p} is $(-1)^n p_n$.
When $E'\ne E$, the graph $(V,E')$ consists of paths.
Let $i_1$ be the order of the path containing $v_1$.
Then $1\le i_1\le n$.
Let $i_2,i_3,\dots$ be the orders of paths counterclockwise in the sequel.
Since the path containing $v_1$ has $i_1$ possibilities:
\[
v_1\dotsm v_{i_1},\quad
v_n v_1\dotsm v_{i_1-1},\quad
v_{n-1}v_n v_1\dotsm v_{i_1-2},\quad
\dots,\quad
v_{n-i_1+1}v_{n-i_1+2}\dotsm v_n v_1,
\]
we can deduce by \cref{prop:csf.p} that
\[
X_{C_n}
=
(-1)^n p_n
+
\sum_{I\vDash n}
i_1\cdotp
(-1)^{(i_1-1)+(i_2-1)+\dotsm}p_{\rho(I)}
=
(-1)^n p_n
+
\sum_{I\vDash n}
i_1
\varepsilon^I
p_{\rho(I)}.
\]
Since $\rho(\Psi^I)=p_I$,
$X_{C_n}$ has the desired analog.
\end{proof}

Now we can produce a ribbon Schur analog of~$X_{C_n}$.

\begin{theorem}\label{thm:ribbon.Schur:cycle}
The chromatic symmetric function of cycles
has a noncommutative analog
\[
\widetilde X_{C_n}
=
\sum_{
I\vDash n,\
i_1=i_{-1}=1,\
i_2,\,\dots,\,i_{-2}\le 2
}
2^{m_1(I)}
\brk3{1-\frac{1}{2^r}}
R_I
-
R_{1^n},
\]
where $i_{-1}$ and $i_{-2}$ are the last and second last part of $I$ respectively, $m_1(I)$
is defined by \cref{def:mk},
and $r$ is the maximum number of parts $1$ that start $I$.
\end{theorem}
\begin{proof}
Recall that $\mathcal H_I$
is the set of ribbons $J$
such that every block in the decomposition $\nabla_I(J)$ is a hook. 
By \cref{Psi2Ribbon}, we can rewrite
the formula in \cref{lem:Psi:cycle} as
\begin{align}
\notag
\widetilde X_{C_n}
&=
(-1)^n
\sum_{J\in\mathcal H_n}
\varepsilon^{n\!J}
R_J
+
\sum_{I\vDash n}
i_1
\varepsilon^I
\sum_{J\in\mathcal H_I}
\varepsilon^{I\!J_1\dotsm J_{\ell(I)}}
R_J
\\
\label{pf:tX.Cn}
&=
\sum_{J\vDash n}
\sum_{J_1\bullet J_2\bullet\dotsm\in\mathcal H(J)}
\abs{J_1}
\varepsilon^{J_1 J_2\dotsm}
R_J
-
\sum_{J\in\mathcal H_n}
\varepsilon^J
R_J,
\end{align}
where $\mathcal H(J)$
is the set of decompositions $J_1\bullet J_2\bullet\dotsm$
such that every block in $J_k$ is a hook.
Here each bullet $\bullet$ is either the concatenation or the near concatenation.
It is direct to compute
\[
[R_{1^n}]\widetilde X_{C_n}
=
\sum_{I\vDash n}
i_1
\varepsilon^{1^{i_1}1^{i_2}\dotsm}
-
\varepsilon^{1^n}
=\sum_{I\vDash n}
i_1
-1
=
n
+
\sum_{j=1}^{n-1}
j\cdotp
2^{n-j-1}
-1
=
2^n-2.
\]
Below we consider $J\vDash n$ such that $J\ne 1^n$.

We introduce a sign-reversing involution
to simplify the inner sum in \cref{pf:tX.Cn}.
Let 
\[
d=J_1\bullet J_2\bullet\dotsm\in\mathcal H(J).
\]
For any box~$\square$ in the ribbon~$J$,
denote 
\begin{itemize}
\item
by~$J_\square$ the hook $J_k$ in $d$ that contains~$\square$, and
\item
by~$\square'$ the box lying to the immediate right of~$\square$, if it exists.
\end{itemize}
We call $\square'$ the \emph{right neighbor} of $\square$.
We say that a box $\square$ of $J$
is an \emph{active} box of $d$
if 
\begin{itemize}
\item
its right neighbor $\square'$ exists, 
\item
$J_\square\ne J_1$, and
\item
the union $J_\square\cup J_{\square'}$ of boxes 
is a hook.
\end{itemize}
Let $\mathcal H'(J)$ be the set 
of decompositions $d\in\mathcal H(J)$ that contain an active box.
We define a transformation $\varphi$ on~$\mathcal H'(J)$ as follows.
Let $d\in\mathcal H'(J)$.
Let $\square$ be the last active box of~$d$.
Define $\varphi(d)$
to be the decomposition obtained from $d$ by
\begin{itemize}
\item
dividing $J_\square$
into two hooks which contain $\square$ and $\square'$ respectively,
if $J_\square=J_{\square'}$;
\item
merging $J_\square$ and~$J_{\square'}$
into a single hook,
if $J_\square\ne J_{\square'}$.
\end{itemize}
From definition, 
we see that $\varphi$ is an involution.
In view of the sign of the inner sum in \cref{pf:tX.Cn},
we define the \emph{sign} of $d=J_1\bullet J_2\bullet\dotsm$ to be 
$\mathrm{sgn}(d)
=\varepsilon^{J_1J_2\dotsm}$.
Then $\varphi$ becomes sign-reversing as
\[
\mathrm{sgn}
\brk1{\varphi(d)}
=
-\mathrm{sgn}(d).
\]
As a result, the contribution of decompositions in $\mathcal H'(J)$
to the inner sum in \cref{pf:tX.Cn} is zero,
and~$\mathcal H(J)$ for the inner sum
can be replaced with the set
\[
\mathcal H''(J)
=\mathcal H(J)\backslash \mathcal H'(J)
\]
of decompositions of $J$ without active boxes.

First of all, we shall show that 
\[
[R_J]\widetilde X_{C_n}=0
\quad\text{if $J$ is a hook and $J\ne 1^n$.}
\] 
Let $J$ be a hook and $J\ne 1^n$.
Let $d\in\mathcal H''(J)$. Then $d$ has no active boxes.
In particular, the second last box $\square$
of~$J$ is not active.
It follows that $J_{\square}=J_1$ and 
\[
\mathcal H''(J)
=
\{J,\,J_1\triangleright 1\},
\]
where $J_1=J\backslash j_{-1}$.
Therefore, by \cref{pf:tX.Cn},
\[
[R_J]\widetilde X_{C_n}
=
n\varepsilon^J
+(n-1)\varepsilon^{J_1 1}
-\varepsilon^J
=0.
\]
Below we can suppose that $J$ is not a hook.
Then the subtrahend in \cref{pf:tX.Cn} vanishes,
and \cref{pf:tX.Cn} implies that
\begin{equation}\label{pf:contribution.J}
[R_J]\widetilde X_{C_n}
=
\sum_{J_1\bullet J_2\bullet\dotsm
\in\mathcal H''(J)}
\abs{J_1}
\varepsilon^{J_1 J_2\dotsm}.
\end{equation}

Second, we claim that 
$[R_J]\widetilde X_{C_n}=0$
unless $j_{-1}=1$.
In fact, if $j_{-1}\ge 2$,
then the second last box of~$J$ is active
for any decomposition $d\in\mathcal H(J)$.
Thus 
\[
\mathcal H''(J)=\emptyset
\quad\text{and}\quad
[R_J]\widetilde X_{C_n}=0.
\]
This proves the claim.
It follows that 
\[
J=1^{s_1}t_1 1^{s_2}t_2\dotsm 1^{s_l}t_l 1^{s_{l+1}},
\quad\text{
where $l\ge 1$, 
$s_1,\dots,s_l\ge 0$, 
$s_{l+1}\ge 1$,
and $t_1,\dots,t_l\ge2$.
}
\] 
Denote the last box on the horizontal part $t_j$
by~$\square_j$. 
We say that a box of~$J$ is a \emph{leader} of 
a decomposition $d\in\mathcal H''(J)$
if it is the first box of some hook of length at least $2$ in $d$.

Third, we claim that 
\[
[R_J]\widetilde X_{C_n}=0
\quad\text{unless $t_2=\dots=t_l=2$}.
\]
Let $j\ge 2$.
If $t_j\ge 3$, then the third last box 
in $t_j$ is active for any $d\in\mathcal H(J)$,
which implies $[R_J]\widetilde X_{C_n}=0$ as before.
This proves the claim.
Moreover, if $\square_j$ is not a leader 
for some $d\in\mathcal H''(J)$,
then the second last box in~$t_j$
is active in $d$, contradicting the choice of $d$.
Therefore, by \cref{pf:contribution.J},
\begin{equation}\label{pf:contribution.J2}
[R_J]\widetilde X_{C_n}
=
\sum_{
\substack{
d=J_1\bullet J_2\bullet\dotsm
\in\mathcal H''(J)\\
\square_j\text{ is a leader of~$d$, }
\forall j\ge 2
}}
\abs{J_1}
\varepsilon^{J_1 J_2\dotsm}.
\end{equation}

Fourth, we shall show that 
\[
[R_J]\widetilde X_{C_n}=0
\quad\text{unless $t_1=2$}.
\]
Suppose that $t_1\ge 3$ and $d=J_1\bullet J_2\bullet \dotsm\in\mathcal H''(J)$. 
Let $B_k$ be the $k$th last box in $t_1$. In particular, $B_1=\square_1$.
We observe that $B_3\in J_1$ since otherwise it would be active.
Moreover,
if $J_1$ ends with $B_3$,
then~$\square_1$ must be a leader of $d$,
since otherwise $B_2$ would be active.
To sum up, we are left to $3$ cases:
\begin{enumerate}
\item
$J_1$ ends with $B_3$,
$J_2=\{B_2\}$,
and $\square_1$ is a leader,
\item
$J_1$ ends with $B_2$,
\item
$J_1$ ends with $B_1$.
\end{enumerate}
Let $h=s_1+t_1$. 
The classification above allows us to transform \cref{pf:contribution.J2} to
\begin{equation}\label{pf:contribution.J3}
[R_J]\widetilde X_{C_n}
=
(h-2)\cdotp
\smashoperator{
\sum_{
\substack{
1^{s_1}(t_1-2)\triangleright 1\triangleright J_3\bullet\dotsm\in\mathcal H''(J)\\ 
\text{$\square_j$ is a leader, }
\forall j\ge 1
}}}
\varepsilon^{1^{s_1}(t_1-2)}
+
(h-1)\cdotp
\smashoperator{
\sum_{
\substack{
J=1^{s_1}(t_1-1)\triangleright J_2\bullet\dotsm\in\mathcal H''(J)\\ 
\text{$\square_j$ is a leader, }
\forall\,j\ge 2}
}}
\varepsilon^{1^{s_1}(t_1-1)}
+
h\cdotp 
\varepsilon^{1^{s_1}t_1}
\smashoperator{
\sum_{
\substack{
J=(1^{s_1}t_1)J_2\bullet \dotsm\in\mathcal H''(J)\\ 
\text{$\square_j$ is a leader, }
\forall\,j\ge 2
}}}
1.
\end{equation}
For $1\le j\le l$,
let $V_j$ be the column of boxes in $J$ that contains $\square_j$.
Then 
\[
\abs{V_j}
=
\begin{cases*}
s_{j+1}+2,
& if $1\le j\le l-1$;\\
s_{j+1}+1,
& if $j=l$.
\end{cases*}
\]
For $j\ge 2$, we observe that $V_j$
is the union of several blocks in $d$.
Conversely, since $\square_j$ is a leader,
$\abs{J_{\square_j}}\ge 2$,
and there are~$2^{\abs{V_j}-2}$ ways to decompose $V_j$
to form the blocks of some $d\in\mathcal H''(J)$.
Computing various cases for $V_1$ in the same vein,
we can deduce from \cref{pf:contribution.J3} that
\begin{align*}
[R_J]\widetilde X_{C_n}
&=
\varepsilon^{1^{s_1}t_1}\brk1{
(h-2)\cdotp 
2^{\abs{s_2\dotsm s_{l+1}}-1}
-
(h-1)\cdotp 
2^{\abs{s_2\dotsm s_{l+1}}}
+h\cdotp 
2^{\abs{s_2\dotsm s_{l+1}}-1}
}
=0.
\end{align*}
Note that each of the $3$ terms in the parenthesis
holds true even for when $l=1$.

Fifth, let us compute the $R_J$-coefficient for 
\[
J
=1^{s_1}2 \dotsm 1^{s_l}2 1^{s_{l+1}},
\quad\text{
where $l\ge 1$, 
$s_1,\dots,s_l\ge 0$, 
and $s_{l+1}\ge 1$.
}
\] 
If $B_1\not\in J_1$,
then $\square_1$ must be a leader,
since otherwise $B_2$ would be active.
Since every vertical hook has sign $1$,
we can deduce from \cref{pf:contribution.J2} that
\begin{align*}
[R_J]\widetilde X_{C_n}
&=
\sum_{
\substack{
1^s J_2\bullet\dotsm\in\mathcal H''(J)\\
1\le s\le s_1\\ 
\text{$\square_j$ is a leader, }
\forall j\ge 1
}}
s
+
\sum_{
\substack{
1^{s_1+1}\triangleright J_2\bullet\dotsm\in\mathcal H''(J)\\
\text{$\square_j$ is a leader, }
\forall j\ge 2
}}
(s_1+1)
-
\sum_{
\substack{
1^{s_1}2 J_2\bullet\dotsm\in\mathcal H''(J)\\ 
\text{$\square_j$ is a leader, }
\forall j\ge 2
}}
(s_1+2).
\end{align*}
Computing the number of decompositions in $\mathcal H''(J)$ 
for each of the $3$ sums above, we derived that
\begin{align*}
[R_J]\widetilde X_{C_n}
&=
\sum_{s=1}^{s_1}
s\cdotp 
2^{s_1-s}
\cdotp
2^{\abs{s_2\dotsm s_{l+1}}-1}
+
(s_1+1)\cdotp 
2^{\abs{s_2\dotsm s_{l+1}}}
-
(s_1+2)\cdotp
2^{\abs{s_2\dotsm s_{l+1}}-1},
\end{align*}
which is true even for $l=1$.
Note that $\abs{s_1\dotsm s_{l+1}}=m_1(J)$, and 
\[
\sum_{s=1}^{s_1}\frac{s}{2^s}
=
2-\frac{s_1+2}{2^{s_1}}
\]
holds as an identity.
Therefore, 
\begin{align*}
[R_J]\widetilde X_{C_n}
&=
2^{m_1(J)-1}
\brk3{
2-\frac{s_1+2}{2^{s_1}}
+\frac{s_1+1}{2^{s_1-1}}
-\frac{s_1+2}{2^{s_1}}}
=
2^{m_1(J)}
\brk3{1-\frac{1}{2^{s_1}}}.
\end{align*}

Finally, collecting the coefficients above, we obtain
\begin{equation}\label{fml:Rpos.cycle}
\widetilde X_{C_n}
=
(2^n-2)R_{1^n}
+
\sum_{
\substack{
J=1^{s_1}2\dotsm 1^{s_l}2 1^{s_{l+1}}\vDash n\\
l\ge 1,\
s_{l+1}\ge 1,\
s_2,\dots,s_l\ge 0
}}
2^{m_1(J)}
\brk3{1-\frac{1}{2^{s_1}}}
R_J,
\end{equation}
which can be recast as the desired formula.
\end{proof}

In view of \cref{fml:Rpos.cycle},
every $R_I$-coefficient is nonnegative.
For instance,
\[
\widetilde X_{C_5}
=
30 R_{1^5}
+4R_{1211}
+6R_{1121}.
\]

\subsection{Tadpoles and their line graphs}

For $m\ge 2$ and $l\ge 0$,
the \emph{tadpole} $T_{m,l}$
is the graph obtained by connecting a vertex on the cycle~$C_m$
and an end of the path $P_l$. 
By definition,
\[
\abs{V(T_{m,l})}
=\abs{E(T_{m,l})}
=m+l.
\]
See \cref{fig:tadpole} for the tadpole $T_{m,l}$
and its line graph $\mathcal L(T_{m,l})$.
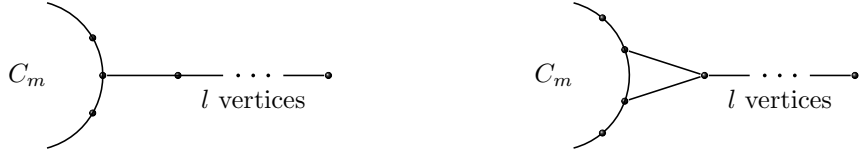
\begin{figure}[h]
\begin{tikzpicture}
\draw (0,0) node {$C_m$};
\node[ball] (c) at (0: \r) {};
\node[ball] at (30: \r) {};
\node[ball] at (-30: \r) {};

\node (p1) at ($ (c) + (\r, 0) $) [ball] {};
\node[ellipsis] at ($ (p1) + (\r-\eps, 0) $) {};
\node[ellipsis] (e2) at ($ (p1) + (\r, 0) $) {};
\node[ellipsis] at ($ (p1) + (\r+\eps, 0) $) {};
\node (p2) at ($ (p1) + (2*\r, 0) $) [ball] {};
\node[below=2pt] at (e2) {$l$ vertices};

\draw[edge] (-75: \r) arc (-75: 75: \r);
\draw[edge] (c) -- ($(p1) + (\r-2*\eps, 0)$);
\draw[edge] (p2) -- ($(p2) - (\r-2*\eps, 0)$);

\begin{scope}[xshift=7cm]
\draw (0,0) node {$C_m$};
\coordinate (c) at (0: \r);
\node[ball] (cu) at (20: \r) {};
\node[ball] (cl) at (-20: \r) {};
\draw[edge] (-75: \r) arc (-75: 75: \r);
\node[ball] at (50: \r) {};
\node[ball] at (-50: \r) {};

\node (p1) at ($ (c) + (\r, 0) $) [ball] {};
\node[ellipsis] at ($ (p1) + (\r-\eps, 0) $) {};
\node[ellipsis] (e2) at ($ (p1) + (\r, 0) $) {};
\node[ellipsis] at ($ (p1) + (\r+\eps, 0) $) {};
\node (p2) at ($ (p1) + (2*\r, 0) $) [ball] {};
\node[below=2pt] at (e2) {$l$ vertices};

\draw[edge] (cu) -- (p1) -- (cl);
\draw[edge] (p1) -- ($(p1) + (\r-2*\eps, 0)$);
\draw[edge] (p2) -- ($(p2) - (\r-2*\eps, 0)$);
\end{scope}
\end{tikzpicture}
\caption{The tadpole $T_{m,l}$ and its line graph $\mathcal L(T_{m,l})$.}\label{fig:tadpole}
\end{figure}
\citet[Theorem~3.1]{LLWY21} pointed out that 
tadpoles possess \citeauthor{GS01}'s
$(e)$-positivity, which implies the $e$-positivity.
They gave the chromatic symmetric function
\begin{equation}\label{X:tadpole:LLWY}
X_{T_{m,l}}
=
(m-1)
X_{P_{m+l}}
-
\sum_{i=2}^{m-1}
X_{C_i}
X_{P_{m+l-i}}
\end{equation}
in their formula (3.11).
By investigating the analog $Y_{\mathcal L(T_{m,l})}\in\mathrm{NCSym}$,
\citet[Theorem~3.2]{WW23-JAC} obtained  
the $(e)$-positivity of the line graphs~$\mathcal L(T_{m,l})$,
which implies the $e$-positivity of the graphs 
$\mathcal L(T_{m,l})$ and $T_{m,l}$.
They \cite[Formulas (3.2) and (3.3)]{WW23-JAC} also obtained 
the formulas
\begin{align}
\label{X:Ltadpole}
X_{\mathcal L(T_{m,l})}
&=
X_{P_l}
X_{C_m}
+
2\sum_{k=0}^{l-1}
X_{P_{k}}
X_{C_{n-k}}
-
2l
X_{P_{n}},
\quad\text{and}\\
\label{X:tadpole}
X_{T_{m,l}}
&=
\frac{1}{2}
\brk1{
X_{\mathcal L(T_{m,l})}
+
X_{P_l}
X_{C_m}
}
=
\sum_{k=0}^{l}
X_{P_{k}}
X_{C_{n-k}}
-lX_{P_{n}}.
\end{align}

\begin{theorem}[Tadpoles]\label{thm:tadpole}
For $0\le l\le n-2$, we have
\[
X_{T_{n-l,l}}
=
\sum_{I\vDash n}
\Theta_I^+(l+1)
w_I
e_I,
\]
where $w_I$ and $\Theta_I^+$ 
are defined by \cref{def:w,def:Theta+}, respectively.
\end{theorem}
\begin{proof}
It is direct 
by \cref{X:tadpole,fml:convolution:path*cycle,X.path}.
\end{proof}

One may deduce \cref{thm:tadpole} alternatively
by using \cref{fml:convolution:cycle*path,X:tadpole:LLWY,sigma-:sigma+}.
The tadpole~$T_{m,1}$ is called an \emph{$m$-pan}.
For example, the  $4$-pan has the chromatic symmetric function 
\[
X_{T_{4,1}}
=
\sum_{I\vDash 5}
\Theta_I^+(2)
w_I
e_I
=
15e_5
+9e_{41}
+3e_{32}
+e_{221}.
\]
We remark that \cref{thm:tadpole} 
reduces to \cref{X.path} when $l=n-2$,
and to \cref{prop:cycle} when $l=0$.

A \emph{lariat} is a tadpole of the form $T_{3,\,n-3}$.
\citet{Dv18} resolved $6$ conjectures of \citet{Wol98} on $X_{T_{3,\,n-3}}$
by analyzing \cref{X:tadpole:LLWY}.
We now bring out a neat formula for $X_{T_{3,\,n-3}}$,
which implies effortless resolutions of the conjectures.

\begin{corollary}[Lariats]\label{thm:lariat}
For $n\ge 3$, we have
$
X_{T_{3,\,n-3}}
=
2\sum_{
I\vDash n,\
i_{-1}\ge 3
}
w_I
e_I$.
\end{corollary}
\begin{proof}
Direct by taking $l=n-3$ in \cref{thm:tadpole}.
\end{proof}

The line graphs of tadpoles also admit simple analogs.

\begin{theorem}[The line graphs of tadpoles]\label{thm:Ltadpole}
For $1\le l\le n-2$,
\[
X_{\mathcal L(T_{n-l,\,l})}
=
\sum_{
I\vDash n,\
\Theta_I^+(l)=0
}
\brk1{\Theta_I^+(l+1)-1}
w_I
e_I
+
2\sum_{
I\vDash n,\
\Theta_I^+(l)\ge 2
}
\Theta_I^+(l+1)
w_I
e_I,
\]
where $w_I$ and $\Theta_I^+$ 
are defined by 
\cref{def:w,def:Theta+}, respectively.
\end{theorem}
\begin{proof}
Let $n=m+l$ and $G=T_{m,l}$.
Taking a noncommutative analog for every term in \cref{X:Ltadpole},
using \cref{wX.path,fml:PathCycle.K,fml:convolution:path*cycle},
we obtain the analog
\begin{align}
\notag
\widetilde X_{\mathcal L(G)}
&=
\sum_{
I\vDash n,\
\Theta_I^+(l)=0
}
\brk1{\Theta_I^+(l+1)+1}
w_I
\Lambda^I
+
2\sum_{I\vDash n}
\brk1{\sigma_I^+(l)-1}
w_I
\Lambda^I
-
2l\sum_{I\vDash n}
w_I
\Lambda^I\\
\label{pf:Ltadpole}
&=
\sum_{
I\vDash n,\
\Theta_I^+(l)=0}
\brk1{\Theta_I^+(l+1)+1}
w_I
\Lambda^I
+2\sum_{I\vDash n}
\brk1{\Theta_I^+(l)-1}
w_I
\Lambda^I.
\end{align}
Let $I\vDash n$ such that $w_I\ne 0$.
We now compute the coefficient $[w_I\Lambda^I]\widetilde X_{\mathcal L(G)}$.
\begin{enumerate}
\item
If $\Theta_I^+(l)=0$, then $\Theta_I^+(l+1)\ge 1$ and
$
[w_I \Lambda^I]
\widetilde X_{\mathcal L(G)}
=
\Theta_I^+(l+1)-1
\ge 0$.
\item
If $\Theta_I^+(l)\ge 1$, then the first sum in \cref{pf:Ltadpole} 
does not contribute, and the second sum gives 
\[
[w_I\Lambda^I]
\widetilde X_{\mathcal L(G)}
=2\brk1{\Theta_I^+(l)-1}
=2\Theta_I^+(l+1)
\]
by \cref{lem:Theta+.t}.
\end{enumerate}
This completes the proof.
\end{proof}

For example, we have
\[
X_{\mathcal L(T_{4,1})}
=
\sum_{
\substack{
I\vDash 5,\
\Theta_I^+(1)=0
}}
\brk1{\Theta_I^+(2)-1}
w_I
e_I
+
2\sum_{
\substack{
I\vDash 5,\
\Theta_I^+(1)\ge 2
}}
\Theta_I^+(2)
w_I
e_I
=
30e_5
+6e_{41}
+6e_{32}.
\]
The noncommutative setting in the proof above is adopted
since $\Lambda^I$-coefficients are considered.

\subsection{Barbells}

For any composition $I=i_1\dotsm i_s\vDash n$,
the \emph{$K$-chain} $K(I)$ is the graph $(V,E)$ where
\begin{align*}
V
&=
\bigcup_{j=1}^s V_j 
\text{ with } 
V_j=\{v_{j1},\
v_{j2},\
\dots,\
v_{ji_j}\},
\quad\text{and}
\\
E
&=
\binom{V_1}{2}
\cup\binom{V_2\cup\{v_{1 i_1}\}}{2}
\cup\binom{V_3\cup\{v_{2 i_2}\}}{2}
\cup\dots
\cup\binom{V_s\cup\{v_{(s-1) i_{s-1}}\}}{2}.
\end{align*}
Here for any set $S$,
\[
\binom{S}{2}
=
\brk[c]1{
\{i,j\}\colon i,j\in S\text{ and }i\ne j}.
\]
See \cref{fig:K-chain}. 
\begin{figure}[h]
\begin{tikzpicture}
\node (c1) at (\r, \r) {$K_{i_1}$};
\draw ($(c1) + (-\r/2, \r)$) arc (90: 270: \r);
\draw ($(c1) + (\r/2, -\r)$) arc (-90: 90: \r);
\node[ellipsis] at ($ (c1) + (0, \r) $) {};
\node[ellipsis] at ($ (c1) + (-\eps, \r) $) {};
\node[ellipsis] at ($ (c1) + (+\eps, \r) $) {};
\node[ellipsis] at ($ (c1) + (0, -\r) $) {};
\node[ellipsis] at ($ (c1) + (-\eps, -\r) $) {};
\node[ellipsis] at ($ (c1) + (+\eps, -\r) $) {};
\node[ball] at ($(c1) + (3*\r/2, 0)$) {};
\coordinate (r1) at ($(c1) - (3*\r/2+\eps, \r+\eps)$);
\draw[densely dotted] (r1) rectangle ($(c1) + (3*\r/2+\eps, \r+\eps)$); 
\path (r1) -- ($(c1) + (3*\r/2+\eps, -\r-\eps)$) node[midway, below=2pt] {$i_1$ vertices};

\node (c2) at ($(c1) + (3,0)$) {$K_{i_2+1}$};
\draw ($(c2) + (-\r/2, \r)$) arc (90: 270: \r);
\draw ($(c2) + (\r/2, -\r)$) arc (-90: 90: \r);
\node[ellipsis] at ($ (c2) + (0, \r) $) {};
\node[ellipsis] at ($ (c2) + (-\eps, \r) $) {};
\node[ellipsis] at ($ (c2) + (+\eps, \r) $) {};
\node[ellipsis] at ($ (c2) + (0, -\r) $) {};
\node[ellipsis] at ($ (c2) + (-\eps, -\r) $) {};
\node[ellipsis] at ($ (c2) + (+\eps, -\r) $) {};
\node[ball] at ($(c2) + (3*\r/2, 0)$) {};
\coordinate (r2) at ($(c2) - (3*\r/2-\eps, \r+\eps)$);
\draw[densely dotted] (r2) rectangle ($(c2) + (3*\r/2+\eps, \r+\eps)$); 
\path (r2) -- ($(c2) + (3*\r/2+\eps, -\r-\eps)$) node[midway, below=2pt] {$i_2$ vertices};

\node (c3) at ($(c2) + (3,0)$) {$K_{i_3+1}$};
\draw ($(c3) + (-\r/2, \r)$) arc (90: 270: \r);
\draw ($(c3) + (\r/2, -\r)$) arc (-90: 90: \r);
\node[ellipsis] at ($ (c3) + (0, \r) $) {};
\node[ellipsis] at ($ (c3) + (-\eps, \r) $) {};
\node[ellipsis] at ($ (c3) + (+\eps, \r) $) {};
\node[ellipsis] at ($ (c3) + (0, -\r) $) {};
\node[ellipsis] at ($ (c3) + (-\eps, -\r) $) {};
\node[ellipsis] at ($ (c3) + (+\eps, -\r) $) {};
\node[ball] at ($(c3) + (3*\r/2, 0)$) {};
\coordinate (r3) at ($(c3) - (3*\r/2-\eps, \r+\eps)$);
\draw[densely dotted] (r3) rectangle ($(c3) + (3*\r/2+\eps, \r+\eps)$); 
\path (r3) -- ($(c3) + (3*\r/2+\eps, -\r-\eps)$) node[midway, below=2pt] {$i_3$ vertices};

\coordinate (c4) at ($(c3) + (3,0)$);
\draw ($(c4) + (-\r/2, \r)$) arc (90: 270: \r);
\draw ($(c4) + (\r/2, -\r)$) arc (-90: 90: \r);
\node[ellipsis] at (c4) {};
\node[ellipsis] at ($ (c4) + (-\eps, 0) $) {};
\node[ellipsis] at ($ (c4) + (+\eps, 0) $) {};
\node[ball] at ($(c4) + (3*\r/2, 0)$) {};

\node (c5) at ($(c4) + (3,0)$) {$K_{i_{-1}+1}$};
\draw ($(c5) + (-\r/2, \r)$) arc (90: 270: \r);
\draw ($(c5) + (\r/2, -\r)$) arc (-90: 90: \r);
\node[ellipsis] at ($ (c5) + (0, \r) $) {};
\node[ellipsis] at ($ (c5) + (-\eps, \r) $) {};
\node[ellipsis] at ($ (c5) + (+\eps, \r) $) {};
\node[ellipsis] at ($ (c5) + (0, -\r) $) {};
\node[ellipsis] at ($ (c5) + (-\eps, -\r) $) {};
\node[ellipsis] at ($ (c5) + (+\eps, -\r) $) {};
\coordinate (r5) at ($(c5) - (3*\r/2-\eps, \r+\eps)$);
\draw[densely dotted] (r5) rectangle ($(c5) + (3*\r/2+\eps, \r+\eps)$); 
\path (r5) -- ($(c5) + (3*\r/2+\eps, -\r-\eps)$) node[midway, below=2pt] {$i_{-1}$ vertices};
\end{tikzpicture}
\caption{The $K$-chain $K(I)$ for $I=i_1i_2\dotsm$.}
\label{fig:K-chain}
\end{figure}
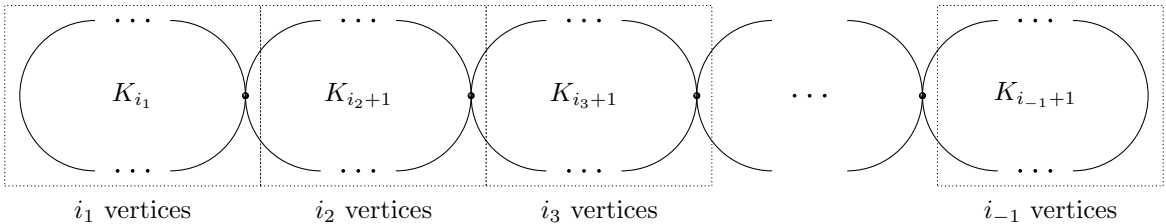
In other words, the $K$-chain $K(I)$
can be obtained from a sequence $G_1=K_{i_1}$,
$G_2=K_{i_2+1}$, $G_3=K_{i_3+1}$, $\dots$, $G_l=K_{i_s+1}$
of cliques such that $G_j$ and $G_{j+1}$ share one vertex,
and that the $s-1$ shared vertices are distinct.
The number of vertices and edges of $K(I)$ are respectively
\[
\abs{V}=n\quad\text{and}\quad
\abs{E}=\binom{i_1}{2}+\sum_{j\ge 2}\binom{i_j+1}{2}.
\]
For instance, 
$K(1^n)=P_n$
and
$K(n)=K_n$.
The family of $K$-chains contains many special graphs.
\begin{enumerate}
\item
A \emph{lollipop} 
is a $K$-chain of the form $K(a1^{n-a})$.
A lariat is a lollipop of the form $K(31^{n-3})$.
\item
A \emph{barbell} 
is a $K$-chain of the form $K(a1^bc)$.
A \emph{dumbbell}
is a barbell of the form $K(a1b)$.
\item
A \emph{generalized bull}
is a $K$-chain of the form $K(1^a 21^{n-a-2})$.
\end{enumerate}

\citet[Theorem 2]{Tom24} gave a formula for the chromatic symmetric function
of melting lollipops, with lollipops as a specialization. 

\begin{theorem}[Lollipops, \citeauthor{Tom24}]\label{thm:lollipop}
Let $n\ge a\ge 1$. Then
$
X_{K(a1^{n-a})}
=
(a-1)!
\sum_{\substack{I\vDash n,\ i_{-1}\ge a}}
w_I
e_I$.
\end{theorem}

\cref{thm:lollipop} covers 
\cref{X.path,X.complete,thm:lariat}.
It can be derived alternatively by 
using \cref{X.complete,X.path} and
\[
X_G
=
(a-1)!
\brk3{
X_{P_n}
-
\sum_{i=1}^{a-2}
\frac{a-i-1}{(a-i)!}
X_{K_{a-i}}X_{P_{n-a+i}}
},
\]
which is due to \citet[Proposition 9]{Dv18}.

Using \citeauthor{Dv18}'s method of
discovering a recurrence relation for
the chromatic symmetric functions of lollipops,
we are able to handle barbells.

\begin{theorem}[Barbells]\label{thm:barbell}
Let $n=a+b+c$, where $a\ge 1$ and $b,c\ge 0$.
Then
\[
X_{K(a1^bc)}
=(a-1)!\
c!
\brk4{
\sum_{
\substack{
I\vDash n,\
i_{-1}\ge a\\
i_1\ge c+1
}}
w_Ie_I
+
\sum_{
\substack{
I\vDash n,\
i_{-1}\ge a\\
i_1\le c< i_2
}}
(i_2-i_1)
\prod_{j\ge 3}(i_j-1)
e_I
},
\]
where $w_I$ is defined by \cref{def:w}.
\end{theorem}
\begin{proof}
Fix $a$ and $n=a+b+c$. See \cref{fig:barbell}.
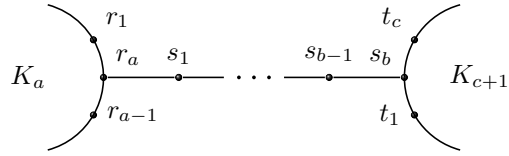
\begin{figure}[htbp]
\begin{tikzpicture}
\node (k1) at (0,0) {$K_a$};
\node (c1) at (0: \r) [ball, label=above right: $r_a$] {};
\node at (30: \r) [ball, label=above right: $r_1$] {};
\node at (-30: \r) [ball, label=right: $r_{a-1}$] {};
\draw[edge] (-75: \r) arc (-75: 75: \r);

\node (p1) at ($ (c1) + (\r, 0) $) [ball, label=above: $s_1$] {};
\node[ellipsis] (e2) at ($ (p1) + (\r, 0) $) {};
\node[ellipsis] at ($ (p1) + (\r-\eps, 0) $) {};
\node[ellipsis] at ($ (p1) + (\r+\eps, 0) $) {};
\node (p2) at ($ (p1) + (2*\r, 0) $) [ball, label=above: $s_{b-1}$] {};
  
\node (c2) at ($ (p2) + (\r, 0) $) [ball, label=above left: $s_b$] {};
\node (k2) at ($ (c2) + (\r,0)$) {$K_{c+1}$};
\node (t1) at ($ (k2)! 1! 30: (c2) $)[ball, label=left: $t_1$] {};
\node (tc) at ($ (k2)! 1! -30: (c2) $)[ball, label=above left: $t_c$] {};

\draw[edge] (c1) -- (p1) -- ($(p1) + (\r-2*\eps, 0)$);
\draw[edge] (c2) -- (p2) -- ($(p2) - (\r-2*\eps, 0)$);
\draw[edge] (c2) arc (180: 180-75: \r);
\draw[edge] (c2) arc (180: 180+75: \r);
\end{tikzpicture}
\caption{The barbell $K(a1^bc)$.}
\label{fig:barbell}
\end{figure}
For $c\in\{0,1\}$, the graph $K(a1^bc)$ reduces to a lollipop,
and the desired formula
reduces to \cref{thm:lollipop}. 
Below we can suppose that $c\ge 2$.
We consider a graph family 
\[
\{G_{b,\,c-k,\,k}\colon k=0,1,\dots,c\}
\]
defined as follows.
Define $G_{b,c,0}=K(a1^bc)$.
For $1\le k\le c$, define $G_{b,\,c-k,\,k}$
to be the graph obtained from~$K(a1^bc)$ 
by removing the edges $s_bt_1$, $\dots$, $s_bt_k$.
In particular, 
\begin{itemize}
\item
$G_{b,\,1,\,c-1}
=K(a1^{b+1}(c-1))$, and
\item
$G_{b,0,c}$ is the disjoint union
of the lollipop $K(a1^b)$ and the complete graph $K_c$.
\end{itemize}
By applying \cref{thm:3del} for the vertex triple 
$(s_b,\,
t_{k+1},\,
t_{k+2})$
in $G_{b,\,c-k,\,k}$, we obtain
\[
X_{G_{b,\,c-k,\,k}}
=2X_{G_{b,c-k-1,k+1}}
-X_{G_{b,c-k-2,k+2}}
\quad\text{for $0\le k\le c-2$}.
\]
Therefore,
one may deduce iteratively that
\begin{align*}
X_{K(a1^bc)}
=X_{G_{b,c,0}}
&=
2X_{G_{b,c-1,1}}
-X_{G_{b,c-2,2}}
=
3X_{G_{b,c-2,2}}
-2X_{G_{b,c-3,3}}
\\
&=
\dotsm
=
cX_{G_{b,1,c-1}}
-(c-1)X_{G_{b,0,c}}
\\
&=
cX_{K(a1^{b+1}(c-1))}
-(c-1)
X_{K(a1^b)}
X_{K_c}.
\end{align*}
Then we can deduce by bootstrapping that
\begin{align*}
X_{K(a1^bc)}
&=
cX_{K(a1^{b+1}(c-1))}
-(c-1)
X_{K(a1^b)}
X_{K_c}
\\
&=
c\brk1{
(c-1)X_{K(a1^{b+2}(c-2))
}
-(c-2)
X_{K(a1^{b+1})}
X_{K_{c-1}}}
-(c-1)
X_{K(a1^b)}
X_{K_c}
\\
%3
&=
c(c-1)\brk1{
(c-2)X_{K(a1^{b+3}(c-3))}
-(c-3)X_{K(a1^{b+2})}X_{K_{c-2}}}
\\
&\qquad
-c(c-2)
X_{K(a1^{b+1})}
X_{K_{c-1}}
-(c-1)
X_{K(a1^b)}
X_{K_c}
\\
&=
\dotsm
\\
&=
c!\,
X_{K(a1^{b+c})}
-
\sum_{i=0}^{c-2}
\frac{c!(c-i-1)}{(c-i)!}
X_{K_{c-i}}
X_{K(a1^{b+i})}.
\end{align*}
By \cref{X.complete,thm:lollipop}, 
we obtain
\begin{equation}\label{pf:X.barbell}
\frac{X_{K(a1^bc)}}{(a-1)! c!}
=
\sum_{
I\vDash n,\
i_{-1}\ge a
}
w_I 
e_I
-
\sum_{i=0}^{c-2}
\sum_{
(c-i)J\vDash n,\
j_{-1}\ge a
}
(c-i-1)
w_J
e_{(c-i)J}.
\end{equation}
We can split it as 
\begin{equation}\label{pf:barbell.Y1Y2}
\frac{X_{K(a1^bc)}}{(a-1)! c!}
=Y_1+Y_2,
\end{equation}
where $Y_1$ is the part containing $e_1$,
and $Y_2$ the part without~$e_1$. 
Let
\begin{align}
\notag
\mathcal W_n
&=
\{i_1i_2\dotsm\vDash n\colon
i_1,i_2,\dots\ge2\},
\\
\label{def:barbell:A}
\mathcal A_n
&=
\{I\in\mathcal W_n\colon 
i_{-1}\ge a,\
i_1\le c\}
\quad\text{and}
\\
\label{def:barbell:B}
\mathcal B_n
&=
\{I\in\mathcal W_n\colon
i_{-1}\ge a,\
i_1\ge c+1
\}.
\end{align}
Then $\mathcal A_n\cap\mathcal B_n=\emptyset$ and 
\[
\mathcal A_n
\sqcup
\mathcal B_n
=
\{I\in\mathcal W_n\colon
i_{-1}\ge a\}.
\]
From \cref{pf:X.barbell}, we obtain
\[
Y_1
=
\sum_{
J\in\mathcal A_{n-1}
\sqcup\mathcal B_{n-1}
}
w_{1J} e_{1J}
-
\sum_{i=0}^{c-2}
\sum_{
(c-i)J\in\mathcal A_{n-1}
}
(c-i-1)
w_{1J}
e_{1(c-i)J}.
\]
Considering $I=(c-i)J$ in the negative part.
When $i$ runs from $0$ to $c-2$
and $J$ runs over compositions 
such that $(c-i)J\in\mathcal A_{n-1}$,
$I$ runs over all compositions in $\mathcal A_{n-1}$. 
Since
\[
(c-i-1)w_{1J}=w_I
\quad\text{and}\quad
e_{1(c-i)J}
=
e_{1I},
\]
we can deduce that
\begin{equation}\label{pf:barbell.Y1}
Y_1
=
\sum_{
J\in
\mathcal A_{n-1}\sqcup
\mathcal B_{n-1}
}
w_{1J}
e_{1J}
-
\sum_{
I\in\mathcal A_{n-1}
}
w_{1I}
e_{1I}
=
\sum_{
J\in\mathcal B_{n-1}
}
w_{1J} 
e_{1J}.
\end{equation}
On the other hand, 
by \cref{pf:X.barbell}, we find
\[
Y_2
=
\sum_{
I\in
\mathcal A_n\sqcup
\mathcal B_n
}
w_I
e_I
-
\sum_{i=0}^{c-2}
\sum_{
(c-i)J\in\mathcal A_n
}
(c-i-1)
w_J 
e_{(c-i)J}.
\]
Similarly, we consider $I=(c-i)J$ in the negative part.
When $i$ runs from $0$ to $c-2$
and $J$ runs over compositions 
such that $(c-i)J\in\mathcal A_n$,
$I$ runs over all compositions in $\mathcal A_n$. 
Note that
\[
(c-i-1)
w_J
=
(i_1-1)
w_{I\backslash i_1}
\quad\text{and}\quad
e_{(c-i)J}
=
e_I,
\]
where $I\backslash i_1=i_2\dotsm i_{-1}$.
Therefore,
\begin{align*}
Y_2
&=
\sum_{I\in\mathcal A_n}
w_I
e_I
+
\sum_{I\in\mathcal B_n}
w_I
e_I
-
\sum_{
I\in\mathcal A_n}
(i_1-1)
w_{I\backslash i_1}
e_I
=
\sum_{
I\in
\mathcal B_n}
w_I
e_I
+
\sum_{
I\in
\mathcal A_n}
f_I
e_I,
\end{align*}
where
\[
f_I
=
w_I-
(i_1-1)
w_{I\backslash i_1}
=
(i_2-i_1)
\prod_{j\ge 3}(i_j-1).
\]
Note that the involution $\phi$ defined
for the compositions $I\in\mathcal A_n$ such that $i_2\le c$
by exchanging the first two parts satisfies $f_{\phi(I)}+f_I=0$.
Therefore,
\[
Y_2
=
\sum_{
I\in\mathcal B_n
}
w_I \cdotp 
e_I
+
\sum_{
I\in\mathcal A_n,\
i_2\ge c+1}
f_I
\cdotp 
e_I.
\]
In view of \cref{def:barbell:A},
the last sum can be recast
by considering the possibility of $i_1=1$ as 
\[
\sum_{
\substack{
I\vDash n,\
i_{-1}\ge a\\
2\le i_1\le c<c+1\le i_2
}}
f_I
\cdotp e_I
=
\sum_{
\substack{
I\vDash n,\
i_{-1}\ge a\\
1\le i_1\le c< i_2
}}
f_I
\cdotp e_I
-
\sum_{
\substack{
J\vDash n-1,\
j_{-1}\ge a\\
j_1\ge c+1
}}
\prod_{k\ge 1}(j_k-1)
\cdotp 
e_{1J},
\]
in which the negative part is exactly $Y_1$ by \cref{pf:barbell.Y1}. 
Therefore, 
\[
Y_2
=
\sum_{
I\in\mathcal B_n}
w_I 
e_I
+
\sum_{
I\vDash n,\
i_{-1}\ge a,\
i_1\le c< i_2}
f_I
\cdotp e_I
-Y_1.
\]
Hence by \cref{pf:barbell.Y1Y2,def:barbell:B},
we obtain the formula as desired.
\end{proof}

For example, 
\begin{align*}
X_{K(31^2 2)}
&=
(3-1)!\
2!
\brk4{
\sum_{
I\vDash 7,\
i_1,i_{-1}\ge 3}
w_I
e_I
+
\sum_{
I\vDash 7,\
i_{-1}\ge 3,\
i_1\le 2< i_2}
(i_2-i_1)
\prod_{j\ge 3}(i_j-1)
e_I}\\
&=
28e_7
+20e_{61}
+12e_{52}
+68e_{43}
+16e_{3^2 1}.
\end{align*}
We remark that \cref{thm:barbell} reduces to \cref{thm:lollipop} 
when $c=0$.
In view of the factor $(i_2-i_1)$ in \cref{thm:barbell},
we do not think it easy to derive \cref{thm:barbell}
by applying \citeauthor{Tom24}'s $K$-chain formula to barbells.
The next two formulas for the graphs 
$K(ab)$ and dumbbells $K(a1b)$ 
are particular cases of \citeauthor{Tom24}'s $K$-chain formula. 
They are straightforward from \cref{thm:barbell}.

\begin{corollary}[\citeauthor{Tom24}]\label{cor:Kab:dumbbell}
Let $a\ge 1$ and $0\le b\le a$.
Then
\begin{align*}
X_{K(ab)}
&=
(a-1)!\,
b!\
\sum_{i=0}^{b}
(a+b-2i)
e_{(a+b-i)i},
\quad\text{and}\\
X_{K(a1b)}
&=
(a-1)!\,b!\
\brk3{
(a-1)(b+1)e_{a(b+1)}
+
\sum_{i=0}^{b}
(a+b+1-2i)e_{(a+b+1-i)i}
}.
\end{align*}
\end{corollary}
\begin{proof}
In \cref{thm:barbell},
taking $n=a+c$ and $b=0$ 
yields the first formula, 
while taking $n=a+1+b$ and $b=1$ 
yields the second.
\end{proof}

We remark that the $e$-positivity of 
the graphs $K(ab)$ and $K(a1b)$ are clear from \cref{prop:epos:alpha<=2}.
On the other hand, in \cref{cor:Kab:dumbbell},
taking $b=1$ in the first formula and
taking $b=0$ in the second result in the same formula 
\[
X_{K(a1)}
=
(a-1)!
\brk1{
(a+1)e_{a+1}
+(a-1)e_{a1}
}.
\]

\subsection{Hats and generalized bulls}
A \emph{hat} is a graph obtained by adding an edge to a path.
Let 
\[
n=a+m+b,\quad\text{where $m\ge 2$ and $a,b\ge 0$}.
\]
The hat $H_{a,m,b}$
is the graph obtained from the path $P_n=v_1\dotsm v_n$
by adding the edge $v_{a+1}v_{a+m}$,
see \cref{fig:hat}. 
It is a unicyclic graph with the cycle length~$m$.
By definition,
\[
\abs{V(H_{a,m,b})}
=\abs{E(H_{a,m,b})}
=n.
\]
It is clear that $H_{a,m,b}$ is isomorphic to $H_{b,m,a}$.
In particular, 
the hat $H_{0,m,b}$ is the tadpole $T_{m,b}$,
the hat $H_{a,2,b}$ is a path with a repeated edge,
and the hat $H_{a,3,b}$ is the generalized bull $K(1^{a+1}21^b)$.

\begin{figure}[h]
\begin{tikzpicture}
\node[ball] (1) at (1, 0) {};
\node[below=2pt] at (1) {$1$};
\node[ball] (2) at (2, 0) {};
\node[ellipsis] (e1) at (3,0) {};
\node[ellipsis] (e1l) at ($ (e1) - (\eps, 0) $) {};
\node[ellipsis] (e1r) at ($ (e1) + (\eps, 0) $) {};
\node[ball] (a+1) at (4, 0) {};
\node[below=2pt] at (a+1) {$a+1$};
\node[ball] (b+2) at (5, 0) {};
\node[ellipsis] (e2) at (6,0) {};
\node[ellipsis] (e2l) at ($ (e2) - (\eps, 0) $) {};
\node[ellipsis] (e2r) at ($ (e2) + (\eps, 0) $) {};
\node[ball] (a+m) at (7, 0) {};
\node[below=2pt] at (a+m) {$a+m$};
\node[ball] (a+m+1) at (8, 0) {};
\node[ellipsis] (e3) at (9,0) {};
\node[ellipsis] (e3l) at ($ (e3) - (\eps, 0) $) {};
\node[ellipsis] (e3r) at ($ (e3) + (\eps, 0) $) {};
\node[ball] (n) at (10, 0) {};
\node[below=2pt] at (n) {$n$};

\draw[edge] (1) -- (2) -- ($ (e1l) - (\eps, 0) $);
\draw[edge] (1) -- (2) -- ($ (e1l) - (\eps, 0) $);
\draw[edge] ($ (e1r) + (\eps, 0) $) -- (a+1) -- (b+2) -- ($ (e2l) - (\eps, 0) $);
\draw[edge] ($ (e2r) + (\eps, 0) $) -- (a+m) -- (a+m+1) -- ($ (e3l) - (\eps, 0) $);
\draw[edge] ($ (e3r) + (\eps, 0) $) -- (n);
\draw[edge] (a+1) arc [start angle=180, end angle=0,
x radius=1.5, y radius=0.5];
\end{tikzpicture}
\caption{The hat $H_{a,m,b}$.}
\label{fig:hat}
\end{figure}
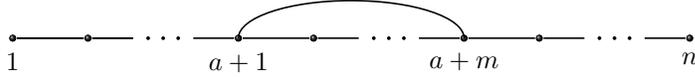

Computing $X_{H_{a,m,b}}$,
we encounter the chromatic symmetric function of spiders with $3$ legs.
For any partition $\lambda=\lambda_1\lambda_2\dotsm\vdash n-1$,
the \emph{spider} $S(\lambda)$ is the tree of order~$n$
obtained by identifying an end of 
the paths $P_{\lambda_1+1}$, $P_{\lambda_2+1}$, $\dots$,
see \cref{fig:Sabc} for an illustration of $S(abc)$.
\begin{figure}[htbp]
\begin{tikzpicture}
\node (O) at (0,0) [ball]{};

\node (a1) at (\r*2, \r) [ball]{};
\node (b1) at (\r*2, 0) [ball]{};
\node (c1) at (\r*2, -\r) [ball]{};

\node (a2) at (\r*3, \r) [ellipsis]{};
\node[above] at (a2) {$a$ vertices};

\node (b2) at (\r*3, 0) [ellipsis]{};
\node[above] at (b2) {$b$ vertices};

\node (c2) at (\r*3, -\r) [ellipsis]{};
\node[above] at (c2) {$c$ vertices};

\node (a2l) at (\r*3-\eps, \r) [ellipsis]{};
\node (b2l) at (\r*3-\eps, 0) [ellipsis]{};
\node (c2l) at (\r*3-\eps, -\r) [ellipsis]{};

\node (a2r) at (\r*3+\eps, \r) [ellipsis]{};
\node (b2r) at (\r*3+\eps, 0) [ellipsis]{};
\node (c2r) at (\r*3+\eps, -\r) [ellipsis]{};

\node (a3) at (\r*4, \r) [ball]{};
\node (b3) at (\r*4, 0) [ball]{};
\node (c3) at (\r*4, -\r) [ball]{};

\draw[edge] (O) -- (a1) -- ($(a2l) + (-\eps, 0)$);
\draw[edge] (O) -- (b1) -- ($(b2l) + (-\eps, 0)$);
\draw[edge] (O) -- (c1) -- ($(c2l) + (-\eps, 0)$);

\draw[edge] ($(a2r) + (\eps, 0)$) -- (a3);
\draw[edge] ($(b2r) + (\eps, 0)$) -- (b3);
\draw[edge] ($(c2r) + (\eps, 0)$) -- (c3);
\end{tikzpicture}
\caption{The spider $S(abc)$, which has $n=a+b+c+1$ vertices.}
\label{fig:Sabc}
\end{figure}
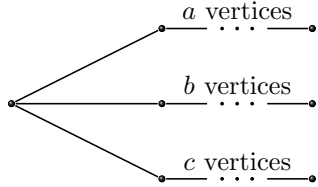
\citet[Lemma~4.4]{Zhe22} showed that for any multiset $\{a,b,c\}$ 
and $n=a+b+c+1$,
\begin{equation}\label{X:spider.abc}
X_{S(abc)}
=
X_{P_n}
+
\sum_{i=1}^c
X_{P_i}
X_{P_{n-i}}
-
\sum_{i=b+1}^{b+c}
X_{P_i}
X_{P_{n-i}}.
\end{equation}

For proving the $e$-positivity of hats,
we introduce a special composition bisection defined as follows.
For any composition $K$
of size at least $b+1$,
we define a \emph{bisection} $K=K_1K_2$ by
\[
\abs{K_1}
=\sigma_K^+(b+1).
\]
It is possible that $K_2$ is empty.
A key property of this bisection is the implication
\begin{equation}\label{pf:hat:implication:K1}
H=K_1H'
\implies
H_1=K_1.
\end{equation}

\begin{theorem}\label{thm:epos:hat}
Every hat is $e$-positive.
\end{theorem}
\begin{proof}
Let $n=a+m+b$.
Since $X_{H_{a,2,b}}=X_{P_n}$ is $e$-positive,
we can suppose that 
$m\ge 3$.
Let $G=H_{a,m,b}$.
When $m\ge 3$, applying \cref{thm:3del}
for the triangle $e_1 e_2 e_3$ in \cref{fig:hat:3del}, 
we obtain 
\begin{equation}\label{rec.Hamb}
X_{H_{a,m,b}}
=X_{H_{a+1,\,m-1,\,b}}
+X_{S(a+1,\,m-2,\,b)}
-X_{P_{a+1}}
X_{T_{m-1,\,b}}.
\end{equation}
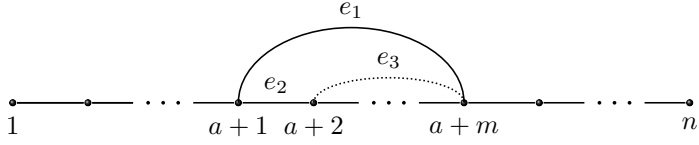
\begin{figure}[h]
\begin{tikzpicture}
\node[ball] (1) at (1, 0) {};
\node[below=2pt] at (1) {$1$};
\node[ball] (2) at (2, 0) {};
\node[ellipsis] (e1) at (3,0) {};
\node[ellipsis] (e1l) at ($ (e1) - (\eps, 0) $) {};
\node[ellipsis] (e1r) at ($ (e1) + (\eps, 0) $) {};
\node[ball] (a+1) at (4, 0) {};
\node[below=2pt] at (a+1) {$a+1$};
\node[ball] (a+2) at (5, 0) {};
\node[below=2pt] at (a+2) {$a+2$};
\node[ellipsis] (e2) at (6,0) {};
\node[ellipsis] (e2l) at ($ (e2) - (\eps, 0) $) {};
\node[ellipsis] (e2r) at ($ (e2) + (\eps, 0) $) {};
\node[ball] (a+m) at (7, 0) {};
\node[below=2pt] at (a+m) {$a+m$};
\node[ball] (a+m+1) at (8, 0) {};
\node[ellipsis] (e3) at (9,0) {};
\node[ellipsis] (e3l) at ($ (e3) - (\eps, 0) $) {};
\node[ellipsis] (e3r) at ($ (e3) + (\eps, 0) $) {};
\node[ball] (n) at (10, 0) {};
\node[below=2pt] at (n) {$n$};

\draw[edge] (1) -- (2) -- ($ (e1l) - (\eps, 0) $);
\draw[edge] (1) -- (2) -- ($ (e1l) - (\eps, 0) $);
\draw[edge] ($ (e1r) + (\eps, 0) $) -- (a+1) -- (a+2) -- ($ (e2l) - (\eps, 0) $);
\draw[edge] ($ (e2r) + (\eps, 0) $) -- (a+m) -- (a+m+1) -- ($ (e3l) - (\eps, 0) $);
\draw[edge] ($ (e3r) + (\eps, 0) $) -- (n);
\draw[edge] (a+1) arc [start angle=180, end angle=0,
x radius=3*\r/2, y radius=\r];
\draw[edge, densely dotted] (a+2) arc [start angle=180, end angle=0,
x radius=\r, y radius=\r/3];

\node[above] at (5.5, \r) {$e_1$};
\node[above] at (4.5, 0) {$e_2$};
\node[above] at (6, \r/3) {$e_3$};
\end{tikzpicture}
\caption{The triangle $e_1e_2e_3$ in
applying the triple-deletion property to the hat $H_{a,m,b}$.}
\label{fig:hat:3del}
\end{figure}
By adding \cref{rec.Hamb} for the parameter $m$ 
from $3$ to the value $m$, 
we obtain
\[
X_G
=
X_{P_n}
+
\sum_{k=1}^{m-2}
\brk1{
X_{S(a+k,\,b,\,m-k-1)}
-X_{P_{a+k}}
X_{T_{m-k,\,b}}
}.
\]
Substituting \cref{X:spider.abc} for spiders into the formula above, 
we deduce that 
\begin{align*}
X_G
&=
X_{P_n}
+
\sum_{k=1}^{m-2}
\brk4{
X_{P_n}
+
\sum_{i=1}^{m-k-1} 
\brk1{
X_{P_i}
X_{P_{n-i}}
-
X_{P_{b+i}}
X_{P_{n-b-i}}}
-
X_{P_{a+k}}
X_{T_{m-k,\,b}}
}\\
&=
\sum_{i=0}^{m-2}
(m-1-i)
X_{P_i}
X_{P_{n-i}}
-
\sum_{i=1}^{m-2}
(m-1-i)
X_{P_{b+i}}
X_{P_{n-b-i}}
-
\sum_{i=1}^{m-2}
X_{T_{m-i,\,b}}
X_{P_{a+i}}.
\end{align*}
Substituting \cref{X.path} for paths
and \cref{thm:tadpole} for tadpoles into it,
we obtain
\begin{align}
\label{pf:X.hat1}
X_G
&=
\smashoperator{
\sum_{
\substack{
K=I\!J\vDash n\\
\abs{I}\le m-2
}}}
\brk1{
m-1-\abs{I}}
w_I
w_J
e_K
-
\smashoperator{
\sum_{\substack{
K=PQ\vDash n\\
b+1
\le\abs{P}
\le b+m-2
}}}
\brk1{
b+m-1-\abs{P}}
w_P
w_Q
e_K
-
\smashoperator{
\sum_{
\substack{
K=PQ\vDash n\\
b+2
\le\abs{P}
\le b+m-1
}}}
\Theta_P^+(b+1)
w_P
w_Q
e_K.
\end{align}
Note that the upper (reps., lower)
bound for $\abs{P}$ in the second (resp., third)
sum can be replaced with $b+m-1$ (resp., $b+1$).
As a consequence, one may think the last two sums 
run as for the same set of pairs $(P,Q)$.
By \cref{lem:sigma+-:Theta+-},
we can merge their coefficients of $w_P w_Q e_K$ as
\begin{align*}
\brk1{
b+m-1-\abs{P}}
+\Theta_P^+(b+1)
&=
m-2-\abs{P}+\sigma_P^+(b+1).
\end{align*}
Therefore, we can rewrite \cref{pf:X.hat1} as
\begin{equation}\label{pf:hat:X}
X_G
=\sum_{
(I,J)\in\mathcal A
}
a_I
w_I
w_J
e_{I\!J}
-
\sum_{(P,Q)\in\mathcal B}
b_P
w_P
w_Q
e_{PQ},
\end{equation}
where 
$a_I=m-1-\abs{I}$,
$b_P=m-2-\abs{P}+\sigma_P^+(b+1)$,
\begin{align*}
\mathcal{A}
&=
\{(I,J)\colon
I\!J\vDash n,\
\abs{I}\le m-2,\
w_I 
w_J\ne 0
\},\quad\text{and}\\
\mathcal{B}
&=
\{(P,Q)\colon 
PQ\vDash n,\ 
b+1
\le\abs{P}
\le b+m-1,\
w_P
w_Q
\ne 0
\}.
\end{align*}
One should note the following facts:
\begin{itemize}
\item
When $(I,J)\in\mathcal{A}$, it is possible that $I=\epsilon$ is the empty composition.
\item
$a_I\ge 1$ for any $(I,J)\in\mathcal{A}$.
\item
$b_P\ge 0$ for any $(P,Q)\in\mathcal{B}$. Moreover, together with \cref{pf:X.hat1},
one may infer that 
\begin{equation}\label{pf:hat:eqrl:bI=0}
b_P=0\iff
\begin{cases*}
\abs{P}=b+m-1\\
\Theta_P^+(b+1)=0
\end{cases*}
\iff
P
=P_1P_2
\text{ with }
\brk1{
\abs{P_1},\,
\abs{P_2}}
=(b+1,\,
m-2).
\end{equation}
\end{itemize}
We will deal with the cases $q_1=1$ and $q_1\ne 1$ respectively. Let 
\begin{align*}
\mathcal B_1
&=
\{(P,Q)\in\mathcal B\colon
q_1=1,\ b_P>0\}\\
&=
\{(P,1Q')\colon 
P1Q'\vDash n,\ 
b+1
\le\abs{P}
\le b+m-1,\
w_P
w_{1Q'}
\ne 0,\
b_P>0
\},
\quad\text{and}\\
\mathcal B_2
&=
\{(P,Q)\in\mathcal B\colon
q_1\ne 1\}\\
&=
\{(P,Q)\colon 
PQ\vDash n,\ 
b+1
\le\abs{P}
\le b+m-1,\
w_{PQ}
\ne 0
\}.
\end{align*}

Let $(P,1Q')\in\mathcal B_1$.
We shall show that the map $h$ defined by
\[
h(P,\,1Q')
=
(1P_2,\,P_1 Q')
\]
is a bijection from $\mathcal B_1$ to the set 
\begin{align*}
\mathcal A_1
&=
\brk[c]1{(1I',\,J)\in \mathcal{A}\colon
\abs{J_2}\ge a}\\
&=
\brk[c]1{(1I',\,J)\colon 
1I'\!J\vDash n,\ 
\abs{1I'}\le m-2,\
w_{1I'}w_J\ne0,\
\abs{J_2}\ge a}.
\end{align*}
Before that,
it is direct to check by definition
that 
\begin{align}
\notag
a_{1P_2}
&=
m-1-\abs{1P_2}
=
b_P,\\
\label{pf:hat.wfwg}
w_{1 P_2}
w_{P_1 Q'}
&=
w_P 
w_{1Q'},\quad\text{and}\\
\notag
e_{1P_2P_1Q'}
&=
e_{P1Q'}.
\end{align}
Therefore, if the bijectivity is proved,
then we can simplify \cref{pf:hat:X} to
\begin{equation}\label{pf:hat:X1}
X_G
=
\sum_{
\substack{
(I,J)\in\mathcal{A},\
i_1\ne 1}}
a_I
w_I
w_J
e_{I\!J}
-
\sum_{
(P,Q)\in\mathcal B_2}
b_P
w_P
w_Q
e_{PQ}
+
\sum_{
(I,J)\in\mathcal A_1'}
a_I
w_I
w_J
e_{I\!J},
\end{equation}
where 
$\mathcal A_1'
=\{(I,J)\in\mathcal{A}\colon i_1=1\}\backslash \mathcal A_1
=\{(I,J)\in\mathcal{A}\colon i_1=1,\ \abs{J_2}\le a-1\}$.

In order to establish the bijectivity of $h$, we need to prove that 
\begin{enumerate}
\item\label[itm]{itm:hat:h:A1}
$h(P,1Q')\in\mathcal A_1$,
\item\label[itm]{itm:hat:h:injectivity}
$h$ is injective, and
\item\label[itm]{itm:hat:h:surjectivity}
$h$ is surjective.
for any $(1I',J)\in\mathcal A_1$,
there exists $(P,1Q')\in\mathcal B_1$ such that 
$h(P,1Q')=(1I',J)$.
\end{enumerate}
We proceed one by one.
\noindent
\cref{itm:hat:h:A1}
If we write $h(P,Q)=(1I',J)$,
then by the implication \eqref{pf:hat:implication:K1},
\begin{equation}\label{pf:hat:P'Q1Q2=I2I1J'}
(I',\,J_1,\,J_2)=(P_2,\,P_1,\,Q').
\end{equation}
Let us check $(1I',J)\in\mathcal A_1$ by definition:
\begin{itemize}
\item
$1I'\!J=1P_2
\cdotp P_1Q'\vDash n$ since $P\cdotp 1Q'\vDash n$;
\item
$\abs{1I'}\le m-2$
since $0<b_P=m-2-\abs{P_2}$;
\item
$w_{1I'}
w_J
=w_{1P_2}
w_{P_1Q'}
=w_P
w_{1Q'}
\ne 0$; and
\item
$\abs{J_2}
=\abs{Q'}
=n-1-\abs{P}
\ge n-1-(b+m-1)
=a$.
\end{itemize}
\cref{itm:hat:h:injectivity}
If 
$h(P,1Q')
=h(\alpha,1\beta')
=(1I',J)$,
then by \cref{pf:hat:P'Q1Q2=I2I1J'},
$P
=P_1
P_2
=J_1 
I'
=\alpha_1
\alpha_2
=\alpha$
and
$Q'
=\beta'$.

\noindent
\cref{itm:hat:h:surjectivity}
Let $(1I',J)\in\mathcal A_1$.
Consider $(P,1Q')=(J_1 I',\,1J_2)$.
By the implication \eqref{pf:hat:implication:K1},
we obtain
\cref{pf:hat:P'Q1Q2=I2I1J'}. Thus
$h(P,Q)
=(1P_2,\,P_1 Q')
=(1I',\,J)$. 
It remains to check that $(P,1Q')\in\mathcal B_1$:
\begin{itemize}
\item
$P1Q'
=J_1I'1J_2\vDash n$
since 
$1I'J\vDash n$.
\item
$b+1
\le \abs{J_1}
\le \abs{J_1 I'}
=\abs{P}
=\abs{J_1 I'}
=n-1-\abs{J_2}
\le 
n-1-a
=b+m-1$.
\item
$w_P 
w_{1Q'}
=w_{J_1I'}
w_{1J_2}
=w_{1I'}
w_J\ne 0$.
\item
If $b_P=0$, then
$b_{J_1 I'}=0$.
By \eqref{pf:hat:implication:K1}
and \eqref{pf:hat:eqrl:bI=0}, 
$\abs{I'}=m-2$, a contradiction.  
Thus $b_P>0$.
\end{itemize}
This proves that $h$ is bijective.

It remains to deal with the case $q_1\ne 1$. 
Continuing with \cref{pf:hat:X1},
we decompose $\mathcal B_2$ as
\[
\mathcal B_2
=\bigsqcup_{
K\in\mathcal K}
\mathcal{B}(K),
\]
where
\begin{align*}
\mathcal K
&=
\{K\vDash n\colon
w_K\ne0,\
\abs{K_1}\le b+m-1\},
\quad\text{and}\\
\mathcal{B}(K)
&=
\{(P,Q)\in
\mathcal B_2\colon
PQ=K\}\\
&=\{(P,Q)\colon
PQ=K,\
b+1\le \abs{P}\le b+m-1\}.
\end{align*}
We remark that the bound restriction in $\mathcal K$ is to guarantee that $\mathcal{B}(K)$ is not trivial:
\[
\abs{K_1}\le b+m-1
\iff
\mathcal{B}(K)\ne\emptyset.
\]
In fact, the restriction implies $(K_1,K_2)\in\mathcal{B}(K)$;
conversely, if $\abs{K_1}\ge b+m$, then
$K$ has no prefix~$P$ such that 
$b+1\le \abs{P}\le b+m-1$.
This proves the equivalence relation.

Now, fix $K\in\mathcal K$. Let
\begin{align*}
\mathcal{A}(K)
&=
\{(I,J)\in\mathcal{A}\colon
i_1\ne 1,\
J_1\overline{I}J_2=K\}\\
&=
\{(I,J)\colon
\abs{I}\le m-2,\
J_1\overline{I}J_2=K
\}.
\end{align*}
Then the sets $\mathcal{A}(K)$ for $K\in\mathcal K$ are disjoint.
In fact, if 
\[
(I,J)\in\mathcal{A}(K)\cap\mathcal{A}(H),
\] 
then $K$ and $H$ have the same prefix $J_1=K_1$ 
by the implication \eqref{pf:hat:implication:K1},
the same suffix $J_2$, 
and the same middle part~$\overline{I}$;
thus $K=H$.
The pairs $(I,J)$ for the first sum in \cref{pf:hat:X1}
that we do not use to cancel the second sum 
form the set 
\begin{align*}
\mathcal A_2
&=
\{(I,J)\in\mathcal{A}
\colon i_1\ne 1\}
\backslash 
\sqcup_{K\in\mathcal K}\mathcal{A}(K)\\
&=
\{(I,J)\in\mathcal{A}
\colon i_1\ne 1,\
J_1\overline{I}J_2\not\in\mathcal K\}\\
&=
\{(I,J)\in\mathcal{A}
\colon i_1\ne 1,\
\abs{J_1}\ge b+m\}.
\end{align*}
Since
$e_{I\!J}=e_K=e_{PQ}$
for any $(I,J)\in\mathcal{A}(K)$
and $(P,Q)\in\mathcal{B}(K)$,
\cref{pf:hat:X1} can be recast as
\begin{equation}\label{pf:hat:X2}
X_G
=
\sum_{K\in\mathcal K}
\Delta(K)e_K
+\sum_{
(I,J)\in\mathcal A_1\cup\mathcal A_2
}
a_I
w_I
w_J
e_{I\!J},
\end{equation}
where
\begin{align*}
\Delta(K)
&=
\sum_{
(I,J)\in\mathcal A(K)
}
a_I
w_I
w_J
-
\sum_{
(P,Q)\in\mathcal{B}(K)
}
b_P
w_P
w_Q.
\end{align*}
Hence it suffices to show that $\Delta(K)\ge 0$.

Let $K_2=m_1m_2\dotsm$.
Then $m_i\ge 2$ for all $i$
since $w_K\ne 0$. 
For $i\ge 0$, 
we define 
\[
P^i
=K_1
\cdotp 
m_1\dotsm m_i,
\quad
Q^i
=m_{i+1}
m_{i+2}
\dotsm,
\quad
I^i
=m_i\dotsm m_1,
\quad\text{and}\quad
J^i
=K_1
\cdotp 
Q^i.
\]
Then $P^i_1=J^i_1=K_1$ 
by the implication \eqref{pf:hat:implication:K1},
\begin{align}
\label{def:hat:l}
\mathcal{B}(K)
&=
\{(P^0,Q^0),\,
\dots,\,
(P^l,Q^l)\},
\quad\text{where $\abs{P^l}
=\sigma_K^-(b+m-1)$, and}\\
\notag
\mathcal{A}(K)
&=
\{(I^0,\,J^0),\
\dots,\
(I^r,\,J^r)\},
\quad\text{
where $\abs{I^r}
=
\sigma_{K_2}^-(m-2)$.
}
\end{align}
We observe that 
\begin{itemize}
\item
$l\le\ell(K_2)-1$, since 
$\abs{Q^l}
=n-\abs{P^l}
\ge n-(b+m-1)=a+1\ge 1$; and
\item
$l\le r$, 
since $\abs{I^l}
=\abs{P^l}-\abs{K_1}
\le (b+m-1)-(b+1)
=m-2$.
\end{itemize}
Therefore,
\begin{equation}\label{DeltaK:hat:Sl}
\Delta(K)
=
S^l
+\sum_{i=l+1}^r 
a_{I^i}
w_{I^i}
w_{J^i},
\end{equation}
where
\[
S^k
=
\sum_{i=0}^k
\brk1{
a_{I^i}
w_{I^i}
w_{J^i}
-b_{P^i}
w_{P^i}
w_{Q^i}}
\quad\text{for $k\ge 0$}.
\]
Let us compare $a_{I^i}$ with $b_{P^i}$,
and compare $w_{I^i}w_{J^i}$ with $w_{P^i}w_{Q^i}$, respectively.
\begin{itemize}
\item
We have
$b_{P^i}=a_{I^i}-1$
for all $0\le i\le l$,
since by \cref{lem:sigma+-:Theta+-},
\[
\abs{P^i}-\sigma_{P^i}^+(b+1)
=\abs{P^i}-\abs{K_1}
=\abs{I^i}.
\]
\item
By \cref{lem:wIJ},
\begin{align}
\notag
w_{P^i}
w_{Q^i}
&=w_K
\cdotp \frac{m_{i+1}}{m_{i+1}-1},
\quad\text{for $0\le i\le l$},
\quad\text{and}\\
\label{pf:hat:wPiwQi}
w_{I^i}
w_{J^i}
&=
\begin{cases*}
w_K, & if $i=0$,\\
\displaystyle
w_K
\cdotp \frac{m_i}{m_i-1}, 
& if $1\le i\le \ell(K_2)$.
\end{cases*}
\end{align}
\end{itemize}
It follows that 
\begin{align}
\label{pf:hat:S0}
S^0
&=
(m-1)w_K
-(m-2)
\cdotp \frac{m_1}{m_1-1}
\cdotp w_K
=w_K
\cdotp\frac{m_1-m+1}{m_1-1},
\quad\text{and}
\\
\label{pf:hat:Sl}
S^l
&=
S^0+w_K\sum_{i=1}^l
\brk3{
\brk1{
m-1-\abs{I^i}}
\cdotp 
\frac{m_i}{m_i-1}
-\brk1{
m-2-\abs{I^i}}
\cdotp 
\frac{m_{i+1}}{m_{i+1}-1}}.
\end{align}
This sum in \cref{pf:hat:Sl} can be simplified by telescoping. 
Precisely speaking, 
since $i$th the negative term
and the $(i+1)$th positive term have sum
\[
-\brk1{
m-2-\abs{I^i}}
\cdotp 
\frac{m_{i+1}}{m_{i+1}-1}
+\brk1{
m-1-\abs{I^{i+1}}}
\cdotp 
\frac{m_{i+1}}{m_{i+1}-1}
=
-m_{i+1},
\]
we can simplify the sum in \cref{pf:hat:Sl}
by keeping the first positive term and the last negative term as
\[
S^l
=S^0
+w_K\brk2{
\brk1{
m-1-\abs{I^1}}
\cdotp \frac{m_1}{m_1-1}
-m_2-\dots-m_l
-\brk1{
m-2-\abs{I^l}}
\cdotp 
\frac{m_{l+1}}{m_{l+1}-1}
}.
\]
Together with \cref{pf:hat:S0}, we can infer that when $l\ge 1$,
\begin{align}
\notag
\frac{S^l}{w_K}
&=
\frac{m_1-m+1}{m_1-1}
+(m-1-m_1)
\cdotp \frac{m_1}{m_1-1}
-m_2-\dots-m_l
-\brk1{
m-2-\abs{I^l}}
\cdotp 
\frac{m_{l+1}}{m_{l+1}-1}
\\
\label{pf:Sl/wK}
&=
\frac{\abs{I^{l+1}}-m+1}{m_{l+1}-1}.
\end{align}
In view of \cref{pf:hat:S0},
we see that \cref{pf:Sl/wK} holds for $l=0$ as well.
Note that
\[
S^l\ge 0
\iff
\abs{I^{l+1}}\ge m-1
\iff	
r=l.
\]
Here we have two cases to deal with.
If $r=l$, then 
\begin{equation}\label{pf:hat:DeltaK:r=l}
\Delta(K)=S^l
=w_K
\cdotp
\frac{\abs{I^{l+1}}-m+1}{m_{l+1}-1}
\ge 0.
\end{equation}
If $r\ge l+1$, 
then 
by \cref{pf:Sl/wK,pf:hat:wPiwQi}, 
\begin{align*}
S^l
+a_{I^{l+1}}
w_{I^{l+1}}
w_{J^{l+1}}
&= 
w_K
\cdotp
\frac{\abs{I^{l+1}}-m+1}
{m_{l+1}-1}
+\brk1{
m-1-\abs{I^{l+1}}}
\cdotp w_K
\cdotp \frac{m_{l+1}}{m_{l+1}-1}
\\
&=w_K
\cdotp 
\brk1{
m-1-\abs{I^{l+1}}}.
\end{align*}
It follows that 
\begin{equation}\label{pf:hat:DeltaK:r>l}
\Delta(K)
=
w_K
\cdotp 
\brk1{
m-1-\abs{I^{l+1}}
}
+
\sum_{i=l+2}^r
a_{I^i}
w_{I^i}
w_{J^i}
\ge 0.
\end{equation}
This completes the proof.
\end{proof}

By carefully collecting all terms of $X_{H_{a,m,b}}$
along the proof of \cref{thm:epos:hat},
and combinatorially reinterpreting 
the coefficients and bound requirements,
we can assemble a positive $e_I$-expansion for
the chromatic symmetric function of hats.

\begin{theorem}[Hats]\label{thm:hat}
Let $n=a+m+b$, where $m\ge 2$ and $a,b\ge0$.
Then
\begin{align}
\label{fml:hat}
X_{H_{a,m,b}}
&=
\sum_{
K\vDash n,\
N_K\le -1}
\frac{-N_K w_K e_K}
{\Theta_K^+(b+m)
+\Theta_K^-(b+m-1)}
+
\sum_{
K\vDash n,\
N_K\ge 1
}
N_K
w_K
e_K\\
\notag
&\qquad
+
\sum_{(I,J)\in\mathcal{S}_{a,m,b}}
\brk1{m-1-\abs{I}}
w_I
w_J
e_{I\!J},
\end{align}
where 
$N_K
=
\Theta_K^+(b+1)
-
\Theta_K^+(b+m)$, 
and if we write $J_1 J_2$ 
as the bisection of $J$
such that $\abs{J_1}=\sigma_J^+(b+1)$,
\begin{align*}
\mathcal{S}_{a,m,b}
&=
\brk[c]1{
(I,J)\colon
K=J_1\overline{I}J_2\vDash n,\
i_1\ne 1,\
\abs{J_1}\le b+m-1,\
2\le\abs{I}\le m-2,\
\abs{J_2}\le \sigma_{\overline{K}}^-(a)-1}\\
&\qquad
\cup\{(I,J)\colon
K=J_1\overline{I}J_2\vDash n,\
i_1\ne 1,\
\abs{J_1}\ge b+m,\
2\le\abs{I}\le m-2\}
\\
&\qquad
\cup
\{(I,J)\colon
I\!J\vDash n,\
i_1=1,\
\abs{I}\le m-2,\
\abs{J_2}\le a-1
\}.
\end{align*}
\end{theorem}
\begin{proof}
We keep notion and notation in the proof of \cref{thm:epos:hat}.
Let $K\in\mathcal K$. Then 
\[
\abs{K_1}\le b+m-1,
\quad\text{i.e.,}\quad
\Theta_K^+(b+1)
\le m-2.
\]
The numerator and denominator  
in \cref{pf:hat:DeltaK:r=l}
can be recast as
\begin{align}
\notag
\abs{I^{l+1}}-m+1
&=
\brk1{
\abs{K_1}+\abs{I^{l+1}}-b-m}
-\brk1{
\abs{K_1}-b-1},
\quad\text{and}\\
\label{pf:hat:numerator}
&=
\Theta_{K}^+(b+m)-\Theta_K^+(b+1),\\
\notag
m_{l+1}-1
&=
\brk1{
\abs{K_1}+\abs{I^{l+1}}-b-m}
+\brk1{
b+m-1-\abs{K_1}-\abs{I^l}}\\
\label{pf:hat:denominator}
&=
\Theta_K^+(b+m)
+\Theta_K^-(b+m-1),
\end{align}
respectively.
By \cref{pf:hat:DeltaK:r=l,pf:hat:DeltaK:r>l},
\begin{align}
\label{pf:hat:sum1}
\smashoperator[r]{
\sum_{
\substack{
K\in\mathcal K\\
\Theta_{K}^+(b+m)\ge \Theta_K^+(b+1)
}}}
\Delta(K)e_K
&=
\sum_{
\substack{
K\vDash n,\ w_K\ne 0\\
\Theta_K^+(b+1)\le m-2\\
\Theta_K^+(b+1)\le \Theta_K^+(b+m)-1
}}
\frac{\Theta_{K}^+(b+m)-\Theta_K^+(b+1)}{\Theta_K^+(b+m)
\Theta_K^-(b+m-1)}
w_K
e_K,
\quad\text{and}\\
\notag
\smashoperator[r]{
\sum_{
\substack{
K\in\mathcal K\\
\Theta_{K}^+(b+m)< \Theta_K^+(b+1)
}}}
\Delta(K)e_K
&=\sum_{
K\in \mathcal K'}
\brk3{
\brk1{\Theta_{K}^+(b+1)-\Theta_K^+(b+m)}
w_K
+
\sum_{i=l+2}^r
a_{I^i}
w_{I^i}
w_{J^i}}
e_K,
\end{align}
where
\[
\mathcal K'
=
\{K\vDash n\colon
w_K\ne0,\
\Theta_{K}^+(b+m)+1
\le 
\Theta_K^+(b+1)
\le m-2
\}.
\]
We claim that the right side of \cref{pf:hat:sum1} 
can be simplified to $K\vDash n$ and 
\begin{equation}\label{pf:fml:hat:sum1}
\Theta_K^+(b+m)
-\Theta_K^+(b+1)
\ge 1.
\end{equation}
In fact, \cref{pf:fml:hat:sum1}
is one of the original bound requirements.
It suffices to show that $\Theta_K^+(b+1)\le m-2$ also holds.
Assume to the contrary that $\Theta_K^+(b+1)\ge m-1$.
Then 
\[
\Theta_K^+(b+1)=\Theta_K^+(b+m)+(m-1)
\]
by the definition \cref{def:Theta+} of $\Theta_K^+$,
contradicting \cref{pf:fml:hat:sum1}.
This proves the claim.

In view of \cref{pf:hat:X2},
it remains to simplify
\[
\sum_{
K\in \mathcal K'}
\sum_{i=l+2}^r
a_{I^i}
w_{I^i}
w_{J^i}
e_K
+\sum_{
(I,J)\in\mathcal A_1\cup\mathcal A_2}
a_I
w_I
w_J
e_{I\!J},
\]
in which the summands have
the same form $a_I w_I w_J e_{I\!J}$.
If a pair $(I,J)$ appears as $(I^i,J^i)$ in the first sum,
then the requirement $i\ge l+2$ is equivalent to say that 
\[
\abs{I}>\abs{I^{l+1}},
\quad\text{i.e.,}\quad
\abs{J_1\overline{I}}
>\sigma_{J_1\overline{I}J_2}^+(b+m),
\]
and the requirement $i\le r$ is equivalent to $\abs{I}\le m-2$.
Thus the set of pairs $(I,J)$ for the first sum is
\begin{align*}
&\bigcup_{
K\in\mathcal K'}
\brk[c]1{
(I^i,J^i)\colon
\overline{I^i} J^i_2=K_2
\text{ for some $l+2\le i\le r$}}\\
=&\
\{
(I,J)\colon
K=J_1\overline{I}J_2\vDash n,\
w_K\ne 0,\
\Theta_K^+(b+m)+1
\le 
\Theta_K^+(b+1)
\le m-2,\\
&\qquad
\abs{J_1\overline{I}}
>\sigma_K^+(b+m),\
\abs{I}\le m-2
\}
\\
=&\
\brk[c]1{
(I,J)\colon
K=J_1\overline{I}J_2\vDash n,\
w_K\ne 0,\
\abs{J_1}\le b+m-1,\
\abs{I}\le m-2,\
\abs{J_1 I}
>\sigma_K^+(b+m)}\\
=&\
\brk[c]1{
(I,J)\colon
K=J_1\overline{I}J_2\vDash n,\
w_K\ne 0,\
\abs{J_1}\le b+m-1,\
\abs{I}\le m-2,\
\abs{J_2}\le \sigma_{\overline{K}}^-(a)-1}.
\end{align*}
On the other hand, 
\begin{align*}
\mathcal A_1
&=
\{(I,J)\in\mathcal{A}\colon
i_1=1,\
\abs{J_2}\le a-1\}\\
&=
\{(I,J)\colon
I\!J\vDash n,\
w_I 
w_J\ne 0,\
i_1=1,\
\abs{I}\le m-2,\
\abs{J_2}\le a-1
\},
\quad\text{and}
\\
\mathcal A_2
&=
\{(I,J)\in\mathcal{A}
\colon i_1\ne 1,\
\abs{J_1}\ge b+m
\}
\\
&=
\{(I,J)\colon
I\!J\vDash n,\
\abs{I}\le m-2,\
w_I 
w_J\ne 0,\
i_1\ne 1,\
\abs{J_1}\ge b+m\}
\\
&=
\{
(I,J)\colon
K=J_1\overline{I}J_2\vDash n,\
w_K\ne 0,\
\abs{J_1}\ge b+m,\
\abs{I}\le m-2
\}.
\end{align*}
Since the product $a_I w_I w_J e_{I\!J}$ vanishes 
when $w_I w_J=0$, 
we can replace the conditions $w_K\ne 0$ for $K=J_1\overline{I}J_2$
with $i_1\ne 1$.
Furthermore,
\begin{align*}
\{
(I,J)\in\mathcal A_2\colon 
I=\epsilon
\}
=
\{
(\epsilon,K)\colon
K\vDash n,\
\abs{K_1}\ge b+m
\}.
\end{align*} 
The sum for $a_I w_I w_J e_{I\!J}$ over this subset can be merged into the 
second sum as
\begin{align*}
&\sum_{
\substack{
K\vDash n\\
\Theta_K^+(b+m)+1
\le \Theta_K^+(b+1)
\le m-2
}}
\brk1{
\Theta_K^+(b+1)
-\Theta_K^+(b+m)
}
w_K 
e_K
+\sum_{
\substack{
K=I\!J\vDash n,\
I=\epsilon\\
\abs{K_1}\ge b+m
}}
a_I
w_I
w_J
e_{I\!J}\\
&\
=
\sum_{
\substack{
K\vDash n\\
\Theta_K^+(b+1)
-\Theta_K^+(b+m)
\ge 1
}}
\brk1{
\Theta_K^+(b+1)
-\Theta_K^+(b+m)
}
w_K
e_K;
\end{align*}
this is because when $\abs{K_1}\ge b+m$,
\[
\Theta_K^+(b+1)
-\Theta_K^+(b+m)
=m-1
=a_\epsilon
\ge 1.
\]
Collecting all the contributions to $X_G$,
we obtain \cref{fml:hat} as desired.
\end{proof}

For example, 
\begin{align*}
X_{H_{1,4,1}}
&=
\sum_{
\substack{
K\vDash 6\\
\Theta_{K}^+(5)
-\Theta_K^+(2)\ge 1
}}
\frac{\Theta_{K}^+(5)-\Theta_K^+(2)}
{\Theta_K^+(5)
+\Theta_K^-(4)}
w_K
e_K
%2
+
\sum_{
\substack{
K\vDash 6\\
\Theta_K^+(2)
-\Theta_K^+(5)\ge 1
}}
\brk1{
\Theta_K^+(2)
-\Theta_K^+(5)
}
w_K
e_K\\
%3
&\qquad
+\smashoperator[r]{
\sum_{
\substack{
J_1\overline{I}J_2\vDash 6,\
i_1\ne 1\\
\abs{J_1}\le 4,\
\abs{I}=2\\
\abs{J_2}
\le \sigma_{\overline{K}}^-(1)-1
}}}
\brk1{3-\abs{I}}
w_I
w_J
e_{I\!J}
%4
+\smashoperator[r]{
\sum_{
\substack{
J_1\overline{I}J_2\vDash 6,\
i_1\ne 1\\
\abs{J_1}\ge 5,\
\abs{I}=2
}}}
\brk1{3-\abs{I}}
w_I
w_J
e_{I\!J}
%5
+\smashoperator[r]{
\sum_{
\substack{
I\!J\vDash 6,\
i_1=1\\
\abs{I}\le 2,\
\abs{J_2}\le0
}}}
\brk1{3-\abs{I}}
w_I
w_J
e_{I\!J}\\
&=
(w_{24}e_{24}/3
+w_{2^3}e_{2^3})
+(w_{42}e_{42}
+w_{132}e_{132}
+3
w_6
e_6
+
3
w_{15}
e_{15})
\\
&\qquad
+0
+0
+2(
w_1
w_5 
e_{51}
+
w_1 
w_{14} 
e_{141}
)\\
&=
18e_{6}+22e_{51}+6e_{42}+6e_{41^2}+2e_{321}+2e_{2^3}.
\end{align*}
Particular hats $H_{a,m,b}$ are special graphs that we explored previously.
\begin{enumerate}
\item
For $a=0$, \cref{thm:hat} reduces to \cref{thm:tadpole} 
since only the second sum in \cref{fml:hat} survives.
\item
For $m=2$, \cref{thm:hat} reduces to \cref{X.path},
since only the first two sums in \cref{fml:hat} survive,
and they are the sum of terms~$w_I e_I$ for $\Theta_K^+(b+1)=0$
and for $\Theta_K^+(b+1)\ge 1$ respectively.
\item
For $b=0$, \cref{thm:hat} may give
a noncommutative analog for the tadpole $T_{m,a}$
that is different from the one given by \cref{thm:tadpole}. 
For instance, these two analogs for $X_{T_{3,2}}$ are respectively
\[
\widetilde X_{H_{2,3,0}}
=10\Lambda^5
+6\Lambda^{14}
+2\Lambda^{23}
+6\Lambda^{32}
\quad\text{and}\quad
\widetilde X_{T_{3,2}}
=10\Lambda^5
+6\Lambda^{14}
+8\Lambda^{23}.
\]
\end{enumerate}

For $m=3$, the hat $H_{a,3,b}$ is the generalized bull $K(1^{a+1}21^b)$.
We produce for generalized bulls a neat formula,
which is not a direct specialization of \cref{thm:hat}. 

\begin{theorem}[Generalized bulls]\label{thm:gbull}
For $a\ge 1$ and $n\ge a+2$,
\[
X_{K(1^{a}21^{n-a-2})}
=
\sum_{
\substack{
I\vDash n,\ 
i_{-1}\ge 3\\
\Theta_I^+(a)\le 1
}}
\frac{i_{-1}-2}
{i_{-1}-1}
\cdotp w_I
e_I
+
\sum_{
\substack{
I\vDash n\\
\Theta_I^+(a)\ge 2
}}
2
w_I
e_I
+
\sum_{
\substack{
J\vDash n-1\\
\Theta_J^+(a)\ge 2
}}
w_J
e_{J1}.
\]
\end{theorem}
\begin{proof}
Let $G=K(1^a21^{n-a-2})$.
Taking $(a,m,b)=(n-a-2,\,3,\,a-1)$ in \cref{fml:hat}, 
we obtain 
\begin{equation}\label{pf:X.bull}
X_G=S_1+S_2+S_3,
\end{equation}
where
\begin{align}
\notag
S_1
&=
\sum_{
K\vDash n,\
N_K\le 0,\
w_K\ne 0}
\frac{-N_K}
{D_K}
w_K
e_K,\\
\notag
S_2
&=
\sum_{
K\vDash n,\
N_K>0,\
w_K\ne 0}
N_K
w_K
e_K,
\quad\text{and}\\
\label{pf:gbull:S3}
S_3
&=
\sum_{
J\vDash n-1,\
\Theta_J^+(a)\ge 2}
w_J
e_{1J},
\end{align}
where
$N_K
=\Theta_K^+(a)
-\Theta_{K}^+(a+2)$
and
$D_K
=\Theta_K^+(a+2)
+\Theta_K^-(a+1)$.

We shall simplify $S_1$ and $S_2$ separately.
For $S_1$, we proceed in $3$ steps.
First, we claim that 
\begin{equation}\label{pf:eqrl:cond:S1}
\begin{cases*}
N_K\le 0\\
w_K\ne 0
\end{cases*}
\iff
\begin{cases*}
\Theta_K^+(a)
\le 1\\
w_K\ne 0.
\end{cases*}
\end{equation}
In fact, for the forward direction,
if $\Theta_K^+(a)
\ge 2$, then 
$N_K=2$
by \cref{lem:Theta+.t}, a contradiction. 
For the backward direction, we have two cases to deal with:
\begin{itemize}
\item
If $\Theta_K^+(a)=0$, 
then $N_K
=-\Theta_{K}^+(a+2)\le 0$ holds trivially.
\item
If $\Theta_K^+(a)=1$,
since $w_K \ne 0$, we then find 
$\Theta_K^+(a+2)\ge 1$ and $N_K\le 0$. 
\end{itemize}
This proves the claim.
It allows us to change the sum range for $S_1$ to 
\[
\mathcal K_a
=
\{
K\vDash n\colon 
w_K\ne 0,\
\Theta_K^+(a)
\le 1
\}
=
\{
K\vDash n\colon 
w_K\ne 0,\
\Theta_K^+(a)+\Theta_K^-(a+1)=1
\}.
\]
Second, for $K\in\mathcal K_a$, 
we have 
$-N_K
=
\Theta_{K}^+(a+2)
-
\Theta_K^+(a)
=
D_K-1$ 
and
\[
S_1
=
\sum_{
K\in\mathcal K_a
}
\frac{D_K-1}
{D_K}
w_K
e_K.
\]
Thirdly,
we claim that 
\begin{equation}\label{pf:gbull:S1:new}
S_1
=
\sum_{
K\in\mathcal K_a
}
\frac{k_{-1}-2}
{k_{-1}-1}
w_K
e_K.
\end{equation}
In fact,
recall from \cref{pf:hat:denominator} that $D_K=m_{l+1}-1$ is a factor 
of $w_K$.
By the definition \cref{def:hat:l} of~$l$,
the part $m_{l+1}$ is the part~$k_j$ of $K$ such that 
\[
\abs{k_1\dotsm k_{j-1}}
=\sigma_K^-(a+1),
\quad\text{i.e.,}\quad
\abs{k_1\dotsm k_j}
=\sigma_K^+(a+2).
\] 
Since $\Theta_K^+(a)\le 1$, we find $j\ge 2$.
For any $K=k_1\dotsm k_s\in\mathcal K_a$,
define $H=\varphi(K)$ to be the composition obtained from $K$ by moving the part $k_j$ to the end, i.e.,
$
H
=
k_1\dotsm 
k_{j-1}
k_{j+1}
\dotsm
k_s
k_j$.
Then $w_H=w_K\ne 0$,
$e_K=e_H$,
and
\[
\abs{h_1\dotsm h_{j-1}}
=\abs{k_1\dotsm k_{j-1}}
=\sigma_K^-(a+1)
\in\{a,\,a+1\}.
\]
Thus $\Theta_H^+(a)\le 1$, and $H\in\mathcal K_a$.
Since $w_H\ne 0$, we find
$\abs{h_1\dotsm h_{j-1}}
=\sigma_H^-(a+1)$.
Therefore, 
$K$ can be recovered from $H$ by moving the last part
to the position immediately after
$h_{j-1}$. 
Hence $\varphi$ is a bijection on $\mathcal K_a$, and
\[
S_1
=
\sum_{
K\in\mathcal K_a}
\frac{D_K-1}
{D_K}
w_K
e_K
=
\sum_{
H\in\mathcal K_a
}
\frac{h_{-1}-2}
{h_{-1}-1}
w_H
e_H.
\]
This proves the claim.
We can strengthen $H
\in\mathcal K_a$
by requiring $h_{-1}\ge 3$
without loss of generality.

Next, the condition $N_K>0$ in $S_2$ can be replaced with $\Theta_K^+(a)\ge 2$
by the equivalence relation \eqref{pf:eqrl:cond:S1}.
Under this new range requirement for $S_2$, we find $N_K=2$ by \cref{lem:Theta+.t}.
Thus 
\begin{equation}\label{pf:gbull:S2:new}
S_2
=\sum_{
K\vDash n,\
\Theta_K^+(a)\ge 2}
2
w_K
e_K.
\end{equation}
Substituting \cref{pf:gbull:S1:new,pf:gbull:S2:new,pf:gbull:S3} into \cref{pf:X.bull}, 
we obtain the desired formula.
\end{proof}

We remark that \cref{thm:gbull} reduces to \cref{thm:lariat} when $n=a+2$.

\appendix
\section{A proof of \cref{prop:cycle}
using the composition method}\label{sec:appendix}

By \cref{Psi2Lambda,epsilon:IJ}, we can deduce from \cref{lem:Psi:cycle} that
\begin{align*}
\widetilde X_{C_n}
&=(-1)^n
\sum_{J\succeq n}
\varepsilon^J
f\!p(J,n)
\Lambda^J
+\sum_{I\vDash n}
\varepsilon^I 
i_1
\sum_{J\succeq I}
\varepsilon^J
f\!p(J,I)
\Lambda^J\\
&=\sum_{J\vDash n}
\brk4{
(-1)^{\ell(J)}
f\!p(J,n)
+\sum_{I\preceq J}
(-1)^{\ell(I)+\ell(J)}
i_1
\cdotp 
f\!p(J,I)}
\Lambda^J.
\end{align*}
Let $J=j_1\dotsm j_t\vDash n$.
Then any composition $I$ of length $s$
that is finer than $J$ can be written as
\[
I
=
(j_{k_1}+\dots+j_{k_2-1})
(j_{k_2}+\dots+j_{k_3-1})
\dotsm
(j_{k_s}+\dots+j_t)
\] 
for some indices $k_1<\dots<k_s$, where $k_1=1$ and $k_s\le t$. 
Therefore,
\begin{align*}
[\Lambda^J]\widetilde X_{C_n}
&=
(-1)^t j_1 
+
\sum_{
1=k_1<\dots<k_s\le t
}
(-1)^{t+s}
\abs{
j_1\dotsm j_{k_2-1}
}
j_1
j_{k_2}
\dotsm 
j_{k_s}\\
&=
j_1 
\brk3{(-1)^t 
+
\sum_{
1=k_1<\dots<k_s\le t
}
(-1)^{t+s}
(j_1 j_{k_2}\dotsm j_{k_s}
+j_2 j_{k_2}\dotsm j_{k_s}
+\dots
+j_{k_2-1}j_{k_2}\dotsm j_{k_s})
}\\
&=
j_1 
\brk3{
(-1)^t
+
\sum_{
1\le h_1<k_2<\dots<k_s\le t
}
(-1)^{t+s}
j_{h_1}
j_{k_2}
\dotsm
j_{k_s}
}\\
&=
j_1
(j_1-1)
(j_2-1)
\dotsm
(j_t-1)
=
(j_1-1)
w_J.
\end{align*}
This proves \cref{prop:cycle}.

\section*{Acknowledgment}
We are grateful to Jean-Yves Thibon, 
whose expectation for noncommutative analogs
of chromatic symmetric functions of graphs 
beyond cycles encourages us to go on the discovery voyage 
for neat formulas.
We would like to thank Foster Tom for pointing us to Ellzey's paper, and for presenting his signed combinatorial formula in an invited talk. 
We also thank the anonymous reviewer 
for sharing his/her understanding of the motivation and for the
writing suggestions.

\bibliography{../csf.bib}

\begin{thebibliography}{34}
\providecommand{\natexlab}[1]{#1}
\providecommand{\url}[1]{\texttt{#1}}
\expandafter\ifx\csname urlstyle\endcsname\relax
  \providecommand{\doi}[1]{doi: #1}\else
  \providecommand{\doi}{doi: \begingroup \urlstyle{rm}\Url}\fi

\bibitem[Aliniaeifard et~al.(2024)Aliniaeifard, Wang, and {van
  Willigenburg}]{AWv24}
F.~Aliniaeifard, V.~Wang, and S.~{van Willigenburg}.
\newblock The chromatic symmetric function of a graph centred at a vertex.
\newblock \emph{Electron. J. Combin.}, 31\penalty0 (4):\penalty0 \#P4.22, 2024.

\bibitem[Corneil et~al.(2010)Corneil, Olariu, and Stewart]{COS10}
D.~G. Corneil, S.~Olariu, and L.~Stewart.
\newblock The {LBFS} structure and recognition of interval graphs.
\newblock \emph{SIAM J. Discrete Math.}, 23\penalty0 (4):\penalty0 1905--1953,
  2010.

\bibitem[Dahlberg and {van Willigenburg}(2018)]{Dv18}
S.~Dahlberg and S.~{van Willigenburg}.
\newblock Lollipop and lariat symmetric functions.
\newblock \emph{SIAM J. Discrete Math.}, 32\penalty0 (2):\penalty0 1029--1039,
  2018.

\bibitem[Dahlberg and {van Willigenburg}(2020)]{Dv20}
S.~Dahlberg and S.~{van Willigenburg}.
\newblock Chromatic symmetric functions in noncommuting variables revisited.
\newblock \emph{Adv. in Appl. Math.}, 112:\penalty0 101942, 25, 2020.

\bibitem[Dahlberg et~al.(2020)Dahlberg, Foley, and {van Willigenburg}]{DFv20}
S.~Dahlberg, A.~Foley, and S.~{van Willigenburg}.
\newblock Resolving {S}tanley's $e$-positivity of claw-contractible-free
  graphs.
\newblock \emph{J. European Math. Soc.}, 22\penalty0 (8):\penalty0 2673--2696,
  2020.

\bibitem[Ellzey(2017)]{Ell17}
B.~Ellzey.
\newblock Chromatic quasisymmetric functions of directed graphs.
\newblock \emph{S\'{e}m. Lothar. Combin.}, 78B:\penalty0 Art. 74, 12, 2017.

\bibitem[Faudree et~al.(1997)Faudree, Flandrin, and Ryj{\'a}{\v c}ek]{FFR97}
R.~Faudree, E.~Flandrin, and Z.~Ryj{\'a}{\v c}ek.
\newblock Claw-free graphs --- a survey.
\newblock \emph{Discrete Math.}, 164\penalty0 (1):\penalty0 87--147, 1997.
\newblock The second Krak{\'o}w conference of graph theory.

\bibitem[Gasharov(1996)]{Gas96P}
V.~Gasharov.
\newblock Incomparability graphs of {$(3+1)$}-free posets are {$s$}-positive.
\newblock In \emph{Proceedings of the 6th {C}onference on {F}ormal {P}ower
  {S}eries and {A}lgebraic {C}ombinatorics ({N}ew {B}runswick, {NJ}, 1994)},
  volume 157, pages 193--197, 1996.

\bibitem[Gasharov(1999)]{Gas99}
V.~Gasharov.
\newblock On {S}tanley's chromatic symmetric function and clawfree graphs.
\newblock \emph{Discrete Math.}, 205\penalty0 (1--3):\penalty0 229--234, 1999.

\bibitem[Gebhard and Sagan(2001)]{GS01}
D.~D. Gebhard and B.~E. Sagan.
\newblock A chromatic symmetric function in noncommuting variables.
\newblock \emph{J. Alg. Combin.}, 13\penalty0 (3):\penalty0 227--255, 2001.

\bibitem[Gelfand et~al.(1995)Gelfand, Krob, Lascoux, Leclerc, Retakh, and
  Thibon]{GKLLRT95}
I.~Gelfand, D.~Krob, A.~Lascoux, B.~Leclerc, V.~Retakh, and J.~Thibon.
\newblock Noncommutative symmetrical functions.
\newblock \emph{Adv. Math.}, 112\penalty0 (2):\penalty0 218--348, 1995.

\bibitem[Guay-Paquet(2013)]{Gua13X}
M.~Guay-Paquet.
\newblock A modular relation for the chromatic symmetric functions of
  (3+1)-free posets.
\newblock arXiv: 1306.2400, 2013.

\bibitem[Haiman(1993)]{Hai93}
M.~Haiman.
\newblock Hecke algebra characters and immanant conjectures.
\newblock \emph{J. Amer. Math. Soc.}, 6\penalty0 (3):\penalty0 569--595, 1993.

\bibitem[Huh et~al.(2020)Huh, Nam, and Yoo]{HNY20}
J.~Huh, S.-Y. Nam, and M.~Yoo.
\newblock Melting lollipop chromatic quasisymmetric functions and {S}chur
  expansion of unicellular {LLT} polynomials.
\newblock \emph{Discrete Math.}, 343\penalty0 (3):\penalty0 111728, 2020.

\bibitem[Li et~al.(2021)Li, Li, Wang, and Yang]{LLWY21}
E.~Y.~H. Li, G.~M.~X. Li, D.~G.~L. Wang, and A.~L.~B. Yang.
\newblock The twinning operation on graphs does not always preserve
  $e$-positivity.
\newblock \emph{Taiwanese J. Math.}, 25\penalty0 (6):\penalty0 1089--1111,
  2021.

\bibitem[MacMahon(1915, 1916)]{Mac1915B}
P.~A. MacMahon.
\newblock \emph{Combinatory Analysis}, volume I and II.
\newblock Camb. Univ. Press, Cambridge, 1915, 1916.

\bibitem[Martin et~al.(2008)Martin, Morin, and Wagner]{MMW08}
J.~L. Martin, M.~Morin, and J.~D. Wagner.
\newblock On distinguishing trees by their chromatic symmetric functions.
\newblock \emph{J. Combin. Theory Ser. A}, 115\penalty0 (2):\penalty0 237--253,
  2008.

\bibitem[Mendes and Remmel(2015)]{MR15B}
A.~Mendes and J.~Remmel.
\newblock \emph{Counting with Symmetric Functions}, volume~43 of \emph{Dev.
  Math.}
\newblock Springer, Cham, 2015.

\bibitem[Orellana and Scott(2014)]{OS14}
R.~Orellana and G.~Scott.
\newblock Graphs with equal chromatic symmetric functions.
\newblock \emph{Discrete Math.}, 320:\penalty0 1--14, 2014.

\bibitem[Shareshian and Wachs(2010)]{SW10}
J.~Shareshian and M.~L. Wachs.
\newblock Eulerian quasisymmetric functions.
\newblock \emph{Adv. Math.}, 225\penalty0 (6):\penalty0 2921--2966, 2010.

\bibitem[Shareshian and Wachs(2012)]{SW12}
J.~Shareshian and M.~L. Wachs.
\newblock Chromatic quasisymmetric functions and {H}essenberg varieties.
\newblock In \emph{Configuration spaces}, volume~14 of \emph{CRM Series}, pages
  433--460. Ed. Norm., Pisa, 2012.

\bibitem[Shareshian and Wachs(2016)]{SW16}
J.~Shareshian and M.~L. Wachs.
\newblock Chromatic quasisymmetric functions.
\newblock \emph{Adv. Math.}, 295:\penalty0 497--551, 2016.

\bibitem[Stanley(1995)]{Sta95}
R.~P. Stanley.
\newblock A symmetric function generalization of the chromatic polynomial of a
  graph.
\newblock \emph{Adv. Math.}, 111\penalty0 (1):\penalty0 166--194, 1995.

\bibitem[Stanley(1998)]{Sta98}
R.~P. Stanley.
\newblock Graph colorings and related symmetric functions: ideas and
  applications: a description of results, interesting applications, \& notable
  open problems.
\newblock \emph{Discrete Math.}, 193\penalty0 (1-3):\penalty0 267--286, 1998.

\bibitem[Stanley(2011. It is a thoroughly revised version of the 1st edition
  published in 1986.)]{Sta11B}
R.~P. Stanley.
\newblock \emph{Enumerative Combinatorics, {V}ol. 1}, volume~49 of
  \emph{Cambridge Stud. in Adv. Math.}
\newblock Camb. Univ. Press, Cambridge, 2nd edition, 2011. It is a thoroughly
  revised version of the 1st edition published in 1986.

\bibitem[Stanley(2023. The main difference with the 1st ed.\ published in 1999
  is 159 additional exercises and solutions on symmetric functions.)]{Sta23B}
R.~P. Stanley.
\newblock \emph{Enumerative Combinatorics. {V}ol. 2}, volume~62 of
  \emph{Cambridge Stud. in Adv. Math.}
\newblock Camb. Univ. Press, Cambridge, 2nd edition, 2023. The main difference
  with the 1st ed.\ published in 1999 is 159 additional exercises and solutions
  on symmetric functions.

\bibitem[Stanley and Stembridge(1993)]{SS93}
R.~P. Stanley and J.~R. Stembridge.
\newblock On immanants of {J}acobi-{T}rudi matrices and permutations with
  restricted position.
\newblock \emph{J. Combin. Theory Ser. A}, 62\penalty0 (2):\penalty0 261--279,
  1993.

\bibitem[Thibon and Wang(2023)]{TW23X}
J.-Y. Thibon and D.~G.~L. Wang.
\newblock A noncommutative approach to the schur positivity of chromatic
  symmetric functions.
\newblock arXiv:2305.07858, 2023.

\bibitem[Tom(2024)]{Tom24}
F.~Tom.
\newblock A signed $e$-expansion of the chromatic symmetric function and some
  new $e$-positive graphs.
\newblock In \emph{Proceedings of the 36th Conference on Formal Power Series
  and Algebraic Combinatorics, S\'{e}m. Lothar. Combin.}, volume 91B, Article
  \#48, 12 pp., Bochum, 2024.

\bibitem[Wang and Wang(2023{\natexlab{a}})]{WW23-DAM}
D.~G.~L. Wang and M.~M.~Y. Wang.
\newblock The $e$-positivity and schur positivity of some spiders and broom
  trees.
\newblock \emph{Discrete Appl. Math.}, 325:\penalty0 226--240,
  2023{\natexlab{a}}.

\bibitem[Wang and Wang(2023{\natexlab{b}})]{WW23-JAC}
D.~G.~L. Wang and M.~M.~Y. Wang.
\newblock The $e$-positivity of two classes of cycle-chord graphs.
\newblock \emph{J. Alg. Combin.}, 57\penalty0 (2):\penalty0 495--514,
  2023{\natexlab{b}}.

\bibitem[Wolfe(1998)]{Wol98}
M.~Wolfe.
\newblock Symmetric chromatic functions.
\newblock \emph{Pi Mu Epsilon J.}, 10\penalty0 (8):\penalty0 643--657, 1998.

\bibitem[Wolfgang~III(1997)]{Wol97D}
H.~L. Wolfgang~III.
\newblock \emph{Two Interactions between Combinatorics and Representation
  Theory: Monomial Immanants and Hochschild Cohomology}.
\newblock PhD thesis, \textsc{MIT}, 1997.

\bibitem[Zheng(2022)]{Zhe22}
K.~Zheng.
\newblock On the $e$-positivity of trees and spiders.
\newblock \emph{J. Combin. Theory Ser. A}, 189:\penalty0 105608, 2022.

\end{thebibliography}
\end{document}